\documentclass[12pt]{amsart}

\usepackage{amssymb}  
\usepackage{latexsym} 
\usepackage{comment}
\usepackage{url}

\DeclareFontEncoding{OT2}{}{} 
\newcommand{\textcyr}[1]{%
 {\fontencoding{OT2}\fontfamily{wncyr}\fontseries{m}\fontshape{n}\selectfont #1}}

\usepackage[all]{xy}

\newcommand{\Sha}{{\mbox{\textcyr{Sh}}}}
\newcommand{\x}{{\bf x}}
\newcommand{\y}{{\bf y}}
\newcommand{\Z}{{\mathbb Z}}
\newcommand{\Q}{{\mathbb Q}}
\newcommand{\R}{{\mathbb R}}

\newcommand{\F}{{\mathbb F}}
\newcommand{\BP}{{\mathbb P}}

\newcommand{\bC}{\bar{C}}

\newcommand{\bK}{\bar{K}}

\newcommand{\bF}{\bar{F}}
\newcommand{\bH}{\bar{H}}

\newcommand{\To}{\longrightarrow}

\newcommand{\Hom}{\operatorname{Hom}}
\newcommand{\Map}{\operatorname{Map}}
\newcommand{\Aff}{\operatorname{Aff}}
\newcommand{\Spec}{\operatorname{Spec}}
\newcommand{\Cov}{\operatorname{Cov}}
\newcommand{\Sel}{\operatorname{Sel}}

\newcommand{\Pic}{\operatorname{Pic}}

\newcommand{\Reg}{\operatorname{Reg}}
\newcommand{\eps}{\varepsilon}

\newcommand{\ran}{r_\text{an}}

\newcommand{\res}{\operatorname{res}}

\newcommand{\diw}{\operatorname{div}}
\newcommand{\Div}{\operatorname{Div}}
\newcommand{\rank}{\operatorname{rank}}

\newcommand{\tors}{{\text{tors}}}

\newcommand{\pr}{\operatorname{pr}}

\newcommand{\HH}{\operatorname{H}}
\newcommand{\Br}{\operatorname{Br}}
\newcommand{\ord}{\operatorname{ord}}

   {\begin{list}{}{\settowidth{\labelwidth}{#1}%
                   \setlength{\leftmargin}{\labelwidth}%
                   \addtolength{\leftmargin}{\labelsep}}}%
   {\end{list}}

\newtheorem{Theorem}{Theorem}[section]
\newtheorem{Lemma}[Theorem]{Lemma}
\newtheorem{Proposition}[Theorem]{Proposition}
\newtheorem{Corollary}[Theorem]{Corollary}

\theoremstyle{definition}
\newtheorem{Definition}[Theorem]{Definition}

\newtheorem{Remark}[Theorem]{Remark}

\numberwithin{equation}{section}

\long\def\authornote#1{%
        \leavevmode\unskip\raisebox{-3.5pt}{\rlap{$\scriptstyle\diamond$}}%
        \marginpar{\raggedright\hbadness=10000
        \def\baselinestretch{0.7}\tiny
        \it #1\par}}

\long\def\authornote#1{\relax}

\begin{document}

\title[Second descent and BSD]{Second Isogeny Descents and the Birch and Swinnerton-Dyer Conjectural Formula}

\author{Brendan Creutz}
\address{School of Mathematics and Statistics, Carslaw Building F07, University of Sydney, NSW 2006, Australia}
\email{brendan.creutz@sydney.edu.au}

\author{Robert L.~Miller}
\address{Mathematical Sciences Research Institute, Berkeley, CA, U.S.A.}
\email{rlm@rlmiller.org}

\date{17 September 2012}

\begin{abstract}
Let $\varphi:E \to E'$ be an isogeny of prime degree $\ell$ between elliptic curves defined over a number field.  We describe how to perform $\varphi$-descents on the nontrivial elements in the Shafarevich-Tate group of $E'$ which are killed by the dual isogeny $\varphi'$. This makes computation of $\ell$-Selmer groups of elliptic curves admitting an $\ell$-isogeny over $\Q$ feasible for $\ell = 5,7$ in cases where a $\varphi$-descent on $E$ is insufficient and a full $\ell$-descent would be infeasible. As an application we complete the verification of the full Birch and Swinnerton-Dyer conjectural formula for all elliptic curves over $\Q$ of rank zero or one and conductor less than $5000$.
\end{abstract}

\maketitle

\renewcommand{\arraystretch}{1.3}


\section{Introduction}
Let $E/k$ be an elliptic curve over a number field $k$. The Mordell-Weil group $E(k)$ of rational points on $E$ is known \cite{Mordell,Weil} to be a finitely generated abelian group. Let $L(E/k,s)$ be the Hasse-Weil $L$-function of~$E$. When $k = \Q$ it is known \cite{modularity1, modularity2} to be an entire function on the complex plane. Its order of vanishing at~$s=1$ is called the {\em analytic rank}, denoted $\ran(E/\Q)$. Birch and Swinnerton-Dyer have conjectured \cite{BSDnotes2} that $\rank(E(\Q)) = \ran(E/\Q)$, that the Shafarevich-Tate group $\Sha(E/\Q)$ is finite and that its order is related to the leading term of the Taylor expansion of $L(E/\Q,s)$ at $s=1$ by a formula recalled below.

For any $\ell \ge 2$, the Mordell-Weil and Shafarevich-Tate groups are also related by an exact sequence of finite abelian groups \[ 0 \to E(k)/\ell E(k) \to \Sel^{(\ell)}(E/k) \to \Sha(E/k)[\ell] \to 0  \,.\] The middle term is the {\em $\ell$-Selmer group of $E$}. Its computation is referred to as an {\em $\ell$-descent on $E$}. This gives an unconditional bound on the Mordell-Weil rank or, when the rank is known (e.g. for elliptic curves over $\Q$ of analytic rank $0$ or $1$), information on the $\ell$-torsion in the Shafarevich-Tate group. A detailed description of how to do an $\ell$-descent when $\ell$ is a prime is given in \cite{SchaeferStoll}. In practice, $\ell$-descents typically require class and unit group information in an extension of $k$ obtained by adjoining the coordinates of one or more nontrivial points in $E[\ell]$. So, even over $\Q$, $\ell$-descents are not usually feasible for primes larger than $3$.

When $E$ admits an isogeny $\varphi:E \to E'$ of degree $\ell$ one can compute Selmer groups associated to $\varphi$ and the dual isogeny $\varphi'$. The two are related by a $5$-term exact sequence (see (\ref{5term}) below), which often allows one to compute the full $\ell$-Selmer group. The advantage to this approach is that, generically, a full $\ell$-descent would require working with an extension of degree $\ell^2-\ell$ (assuming $E$ admits an $\ell$-isogeny), whereas the $\varphi$- and $\varphi'$-Selmer groups can be determined from class and unit group information in extensions of degree $\ell-1$. In many cases, the $\varphi$- and $\varphi'$-Selmer groups can actually be determined with very little explicit computation. For details in various specific cases, the reader may wish to consult \cite{Bandini,CohenPazuki,FisherThesis,FisherJEMS,FlynnGrattoni,Selmer,Silverman,Stephens}. General treatments are given in \cite{SchaeferStoll} and \cite{MillerStoll}, the latter also containing a recent and rather thorough review of the existing literature. The disadvantages are that this does not apply to general elliptic curves and that, even when it does, it may fail to yield sufficient information to compute the $\ell$-Selmer group.

The latter issue can be dealt with if one can determine the subgroup \[\varphi\left(\Sha(E/k)[\ell]\right) \subset \Sha(E'/k)[\varphi']\,.\] In principle this can be achieved by computing the Cassels pairing \cite{CasselsIV} on $\Sha(E'/k)[\varphi']\times\Sha(E'/k)[\varphi']$. When $E[\ell] \simeq \Z/\ell\Z\times \mu_\ell$ as a Galois module, the pairing can be evaluated by writing it as a sum of local pairings \cite{FisherCTP}. However, such splittings do not occur generically; over $\Q$ this is only possible for $\ell \le 5$. In general one can try to reduce to the split case by passing to an appropriate extension \cite[Section 2.5]{FisherCTP}, but there is no guarantee that this will yield the required information over $\Q$.

In this paper we give an alternative method for computing the subgroup $\varphi\left(\Sha(E/k)[\ell]\right)$, analogous to the approach for computing $\ell^2$-Selmer groups developed in \cite{CreutzThesis,Creutz2ndp}. Given $C \in \Sha(E'/k)[\varphi']$ we compute a finite set, called the {\em $\varphi$-Selmer set of $C$}, which consists of certain everywhere locally solvable coverings of $C$. This set is nonempty precisely when $C$ admits a lift to $\Sha(E/k)[\ell]$. We refer to our method as a {\em second $\varphi$-descent}. Together with $\varphi$- and $\varphi'$-descents on $E'$ and $E$, this always allows one to determine the $\ell$-Selmer group. All of the descents involved require class and unit group information in extensions of degree at most $\ell$, making computation of $\ell$-Selmer groups of general elliptic curves admitting an $\ell$-isogeny over $\Q$ feasible in practice for $\ell = 5,7$.

As an application we complete the proof of the following theorem.
\begin{Theorem}
\label{MainTheorem}
Let $E$ be an elliptic curve over $\Q$ of conductor $N < 5000$ and such that $\ran(E/\Q) \leq 1$. Then the full Birch and Swinnerton-Dyer conjecture holds for $E$. This means that $\ran(E/\Q)$ is equal to the rank of $E(\Q)$, $\Sha(E/\Q)$ is finite and that
\[
 \frac{L^{(r)}(E/\Q, 1)}{r!} = \frac{\Omega(E) \cdot \prod_pc_p(E) \cdot \Reg(E(\Q)) \cdot \#\Sha(E/\Q)}{\#E(\Q)_\tors^2}\,,
\] where $c_p(E)$ denotes the Tamagawa number at the prime $p$, $\Reg(E(\Q))$ is the regulator of $E$ and $\Omega(E)$ is the integral over $E(\R)$ of the absolute value of the minimal invariant differential of $E$.
\end{Theorem}

Several remarks are in order. Firstly we note that the bound of $5000$ on the conductor is an arbitrary one, but it seemed to provide a good balance between challenge and feasibility. The restriction on the rank is less so. If $\ran(E/\Q) \leq 1$ then it is known \cite{Kolyvagin, modularity1, modularity2} that $\ran(E/\Q)$ is equal to the rank of $E(\Q)$ and $\Sha(E/\Q)$ is finite. However, there is no curve of rank greater than one for which $\Sha(E/\Q)$ is known to be finite. The rank conjecture has been verified by John Cremona \cite{CremonaDB} for many individual curves with $\ran(E/\Q) \leq 3$, in particular for all curves of conductor up to 130,000. Thus the hypothesis in Theorem \ref{MainTheorem} that $\ran(E/\Q) \leq 1$ can be stated in terms of the algebraic rank instead. However, there is no curve of algebraic rank greater than three for which the rank conjecture is known, as there is no known method for proving that the analytic rank is what it appears to be in such cases.

Theorem \ref{MainTheorem} extends the efforts begun in \cite{bsdalg}, where the authors proved the $\ell$-part of BSD (i.e. the claim that the order of $\Sha(E/\Q)$ predicted by the formula is a rational number and that the exponent of $\ell$ in its prime factorization is the same as the exponent of the actual order of $\Sha(E/\Q)$) for such curves of conductor up to 1000 without complex multiplication, for $\ell = 2$ and for $\ell$ such that the mod-$\ell$ Galois representation attached to~$E$ is irreducible and $\ell \nmid \prod_p c_p(E)$. In \cite{RMiller} the second author extended this to such curves of conductor up to 5000 for $\ell = 2,3$ and for all $\ell \ge 5$ such that the mod-$\ell$ representation is irreducible, regardless of complex multiplication or Tamagawa numbers.

The restriction to irreducible mod $\ell$ representations above owes itself to a result of Kolyvagin \cite{Kolyvagin} which allows one to obtain upper bounds for $\ord_\ell(\#\Sha(E/\Q))$ through suitable Heegner index computations. Neither Kolyvagin's original result nor any of its various extensions (e.g. \cite[3.2 -- 3.5]{bsdalg} or \cite[5.1 -- 5.4]{RMiller}) yield effective upper bounds when $E$ has CM and reducible mod $\ell$ representation or when $\ell \mid \#E'(\Q)_{\tors}$ for some $\Q$-isogenous curve $E'$. The remaining CM curves were dealt with in \cite{MillerStoll} by means of $\ell$-isogeny descent, using knowledge of class groups of cyclotomic fields to avoid much of the explicit computation. In the remaining cases $\ell \mid \#E'(\Q)_{\tors}$, so $\ell \le 7$. For the majority of these the $\ell$-primary part of $\Sha$ was computed in \cite{FisherThesis,FisherJEMS} by an $\ell$-isogeny descent. However, for the eleven pairs $(E, \ell)$ listed in Table \ref{cases} (the first entry in each pair is an isogeny class labeled as in the Cremona database \cite{CremonaDB}), this was insufficient. We use second $\ell$-isogeny descents to compute $\Sha(E/\Q)[\ell^\infty]$ for these $11$ remaining cases and complete the proof of Theorem \ref{MainTheorem}.

\begin{Remark}
Some of the $11$ pairs in Table \ref{cases} may have been handled independently by other authors using alternative means. We note in particular that (570l,5) has been dealt with by computation of the Cassels-Tate pairing \cite[sec. 2.5]{FisherCTP}.
\end{Remark}

\begin{table}
\caption{Remaining isogeny classes}
\label{cases}
\begin{center}
\begin{tabular}{|lr|lr|}\hline
$E$  & $\ell$ & $E$   & $\ell$ \\\hline
546f &  7  & 1938j &  5 \\
570l &  5  & 1950y &  5 \\
858k &  7  & 2370m &  5 \\
870i &  5  & 2550be&  5 \\
1050o&  5  & 3270h &  5 \\
1230k&  7  & &\\\hline
\end{tabular}
\end{center}
\end{table}

\subsection{Organization}
Section \ref{phicoverings} gives definitions and basic properties of the {\em $\varphi$-coverings} we aim to compute. In Section  \ref{DescentMap} we define a {\em descent map} which ultimately gives a more concrete realization of these abstractly defined objects. We then develop a cohomological description of the descent map and use this to derive several important properties in Section \ref{Cohomology}. In Section \ref{Geometry} we show how to write down explicit models in projective space for the coverings parametrized by our descent. In particular, this gives an explicit inverse to the descent map.

In Section \ref{ComputingSel} we give an algorithm for computing the $\varphi$-Selmer set of an element in $\Sha(E'/k)[\varphi]$. With the material of the preceding sections this is a rather standard reduction to computational algebraic number theory. Following this we introduce a fake Selmer set which is easier to compute, but may differ from the genuine Selmer set slightly. We conclude with the proof of Theorem \ref{MainTheorem} and some explicit examples. In two of these we use a second $\ell$-isogeny descent to determine the $\ell$-primary part of the Shafarevich-Tate group with $\ell = 5$ or $7$. In a third, we give an example where a second isogeny descent is needed to compute the analytic order of $\Sha$. Namely, we find a generator of the Mordell-Weil group of canonical height approximately $242$ by writing down an explicit model for a degree $25$ covering of $E$, and use this to compute the regulator. This last example makes essential use of recent work by Fisher on minimization and reduction of genus one normal curves of degree $5$ \cite{FisherMinRed5}.

\subsection{Notation}
Throughout $\ell$ will denote an odd prime, $n$ will denote an integer and $K$ will denote a perfect field of characteristic not dividing $n\ell$ with algebraic closure $\bK$ and with absolute Galois group $G_K$. The symbol $k$ will denote a number field; the completion of $k$ at a prime $v$ will be denoted $k_v$.

For a projective curve $C/K$ and a commutative $K$-algebra $A$ we use $C \otimes_K A$ to denote the extensions of scalars $C \times_{\Spec K}\Spec A$. In the special case that $A = \bK$ this will also be denoted by $\bC$. The function field of $C$ is denoted $\kappa(C)$. We use $\Div(C)$ to denote the group of $K$-rational divisors on $C$. If $P \in C(\bK)$ is a point, the corresponding divisor in $\Div(\bC)$ will be denoted $[P]$. The quotient of $\Div(C)$ by the subgroup of $K$-rational principal divisors is denoted by $\Pic(C)$. The group of $K$-rational divisor classes on $C$ is $\Pic(\bC)^{G_K}$. We remind the reader that the obvious map $\Pic(C) \to \Pic(\bC)^{G_K}$ is injective, but may fail to be surjective.

\section{$\varphi$-coverings}
\label{phicoverings}
A {\em $K$-torsor under $E$} is a smooth, projective curve $C/K$ together with algebraic group action of $E$ on $C$ which is defined over $K$ and is simply transitive on $\bK$-points (we will always assume the action of $E$ is fixed even if it is not explicitly given in the notation). Any point $P \in C(\bK)$ gives an isomorphism (defined over the field of definition of $P$) $\psi_P: C \simeq E$ sending a point $Q \in C(\bK)$ to the unique $R \in E(\bK)$ such that $Q = P+R$. We say that an isomorphism of curves $\psi:C\simeq E$ is {\em compatible with the torsor structure} if it is of this type. The $K$-isomorphism classes of $K$-torsors under $E$ are parameterized by the Weil-Ch\^atelet group, $\HH^1(K,E)$. A $K$-torsor under $E$ is trivial (i.e. isomorphic to $E$ acting on itself by translations) if and only if it has a $K$-rational point. Hence, for a number field $k$, the Tate-Shafarevich group, \[ \Sha(E/k) = \ker\left(\HH^1(k,E) \to \bigoplus \HH^1(k_v,E)\right)\,,\] parameterizes isomorphism classes of everywhere locally solvable torsors. 

\begin{Definition}
Let $\varphi:E\to E'$ be an isogeny of elliptic curves defined over $K$ of degree not divisible by the characteristic of $K$, and let $C$ be a $K$-torsor under $E'$. A {\em $\varphi$-covering} of $C$ is a morphism of curves $D \stackrel{\pi}{\to} C$ defined over $K$ which fits into a commutative diagram \[ \xymatrix{ D \ar[d]^\pi\ar[r]^{\psi_D} & E \ar[d]^\varphi \\ C \ar[r]^{\psi_C} & E'}\] where $\psi_C$ and $\psi_D$ are isomorphisms of curves defined over $\bK$, with $\psi_C$ compatible with the torsor structure on $C$. Two $\varphi$-coverings of $C$ are isomorphic if they are $K$-isomorphic as $C$-schemes. The set of all $K$-isomorphism classes of $\varphi$-coverings of $C$ defined over $K$ is denoted by $\Cov^{(\varphi)}(C/K)$. If $K = k$ is a number field, we define the {\em $\varphi$-Selmer set of $C$}, denoted $\Sel^{(\varphi)}(C/k)$, to be the set of all isomorphism classes of $\varphi$-coverings of $C$ which are everywhere locally solvable.
\end{Definition}

Since the possible choices for $\psi_C$ differ by translations and $\varphi$ is surjective on $\bK$-points, all $\varphi$-coverings of $C$ are $\bK$-isomorphic as $C$-schemes. Geometrically they are Galois coverings of $C$ with group isomorphic to $E[\varphi]$. So by the twisting principle $\Cov^{(\varphi)}(C/K)$ is, if nonempty, a principal homogeneous space for $\HH^1(K,E[\varphi])$, with the action given by twisting. Every $\varphi$-covering $D \stackrel{\pi}\to C$ comes equipped with the structure of a $K$-torsor under $E$. Indeed, any isomorphism $\psi_D$ as in the definition gives an action of $E$ on $D$ via $D\times E \ni (Q,P) \mapsto \psi_D^{-1}(\psi_D(Q) + P) \in D$. The isomorphism class of the torsor does not depend on $\psi_D$ as we have assumed that $\psi_C$ is compatible with the torsor structure on $C$. 

The map $\varphi:E \to E'$ gives $E$ the structure of a $\varphi$-covering of $E'$ (considered as the trivial torsor). Considering the elements of $\Cov^{(\varphi)}(E'/K)$ as twists of this canonical element gives a canonical identification $\Cov^{(\varphi)}(E'/K) = \HH^1(K,E[\varphi])$. Thus, $\Cov^{(\varphi)}(E'/K)$ is an abelian group.

The $\varphi$-Selmer set is finite and, at least in principle, computable \cite{ChevalleyWeil}. We refer to its computation as a {\em $\varphi$-descent on $C$}. This also applies when $C = E'$. In this case the $\varphi$-Selmer set is a finite abelian group, and it sits in an exact sequence
\begin{align}
\label{DescentSequence}
0 \to E'(k)/\varphi E(k) \to \Sel^{(\varphi)}(E'/k) \to \Sha(k,E)[\varphi] \to 0\,
\end{align} (see \cite[Theorem X.4.2]{Silverman}).

\begin{Remark} The reader is cautioned that the middle term in (\ref{DescentSequence}) is almost always referred to as the $\varphi$-Selmer group of $E$ (with $E$ rather than $E'$ present in the notation). This is reasonable given that it is a subgroup of $\HH^1(k,E[\varphi])$. However, we find it convenient to adopt this nonstandard notation since in our interpretation of coverings it is the codomain of the covering map which plays the primary role.
\end{Remark}

Our interest in $\varphi$-descents stems from the following relation to $\varphi$-divisibility in the Shafarevich-Tate group.

\begin{Lemma}
\label{phidiv}
Suppose $C$ is a $k$-torsor under $E'$ defined over a number field $k$. Then $C \in \varphi\Sha(E/k)$ if and only if $\Sel^{(\varphi)}(C/k) \ne \emptyset$.
\end{Lemma}

\begin{proof}
We may assume $C \in \Sha(E'/k)$, otherwise the statement is trivial. Suppose $C$ is killed by $n$ and consider the following commutative diagram. \[ \xymatrix{ \Sel^{(n\circ\varphi)}(E'/k) \ar[d]^{\varphi_*} \ar[r] & \Sha(E/k)[n\circ\varphi] \ar[r]\ar[d]^\varphi& 0\\ \Sel^{(n)}(E'/k) \ar[r] & \Sha(E'/k)[n] \ar[r]& 0 }\] By the exact sequence (\ref{DescentSequence}), $C$ admits a lift to an $n$-covering $C \stackrel{\pi}\to E'$ in the $n$-Selmer group of $E'$. Each choice of lift gives a map $\Sel^{(\varphi)}(C/k) \ni (D,\rho) \mapsto (D,\pi\circ\rho) \in \Sel^{(n\circ\varphi)}(E'/k)$. The image of this map is exactly the fiber above $(C,\pi)$ under the map denoted $\varphi_*$ in the diagram above. From this one deduces the result from commutativity and the fact that the horizontal maps are surjective.
\end{proof}

The dual isogeny $\varphi':E' \to E$ satisfies $\varphi\circ\varphi' = \deg(\varphi)$. There is also a Selmer group associated to $\varphi'$. It is related to the $\varphi$- and $\deg(\varphi)$-Selmer groups by the 5-term exact sequence \cite[Lemma 6.1]{SchaeferStoll},
\begin{align}
\label{5term}
0 \to \frac{E'(k)[\varphi']}{\varphi \left(E(k)[\deg(\varphi)]\right)} \to& \Sel^{(\varphi)}(E'/k) \to \Sel^{(\deg(\varphi))}(E/k) \\\notag
&\to \Sel^{(\varphi')}(E/k) \to \frac{\Sha(E'/k)[\varphi']}{\varphi\left(\Sha(E/k)[\deg(\varphi)]\right)} \to 0\,.
\end{align}

It is also worth noting that the order of $\Sel^{(\varphi')}(E/k)$ can be computed from the order of $\Sel^{(\varphi)}(E'/k)$ using a formula of Cassels in \cite{CasselsVIII}, and vice versa. Lemma \ref{phidiv} shows that one can compute the final term in (\ref{5term}) by doing $\varphi$-descents on the elements of $\Sha(E'/k)[\varphi']$. Together with the $\varphi$- and $\varphi'$-descents on $E'$ and $E$, this allows for computation of the $\deg(\varphi)$-Selmer group. Since the elements of $\Sha(E'/k)[\varphi']$ would presumably be obtained by the $\varphi'$-descent on $E$ we refer to this as a {\em second isogeny descent on $E$}.

\begin{Remark}
\label{phidivremark}
For $C \in \Sha(E'/k)$ it is well known that the condition in Lemma \ref{phidiv} is equivalent to requiring that $C$ pair trivially with all elements of $\Sha(E'/k)[\varphi']$ under the Cassels pairing defined in \cite{CasselsIV}. The fact that the pairing is alternating implies that the order of the final term in (\ref{5term}) is a square whose prime factors divide $\deg(\varphi)$.
\end{Remark}

\subsection{$\ell$-isogeny coverings }
For the remainder of the paper we assume that $\deg(\varphi)$ is an odd prime $\ell$ and let $C \stackrel{\pi}{\to} E$ be a $\varphi'$-covering of an elliptic curve $E$ defined over $K$. By definition there is an isomorphism $\psi_C:C \simeq E'$ such that $\pi = \varphi'\circ \psi_C$, which gives $C$ the structure of a $K$-torsor under $E'$. Let $X = \pi^{-1}(0_{E})$ denote the set of points lying above the identity on $E$. The action of $E'$ on $C$ restricts to an action of $E'[\varphi']$ on $X$. Moreover, $\pi^*[0_E]$ is a $K$-rational divisor of degree $\ell$ on $C$. The linear system corresponding to $\pi^*[0_E]$ gives an embedding of $C$ in $\BP^{\ell-1}$ as a genus one normal curve of degree $\ell$. We remind the reader that a {\em genus one normal curve of degree $n \ge 3$} defined over $K$ is a smooth projective curve of genus one embedded in $\BP^{n-1}$ via the complete linear system associated to some $K$-rational effective divisor of degree $n$ (see \cite[I, Section 1.3]{CFOSS}).

If $D \stackrel{\rho}\to C$ is a $\varphi$-covering of $C$, then composing the covering maps gives $D$ the structure of $\ell$-covering of $E$. This results in a map (which depends on both $C$ and the map $\pi$) \[ \Psi_\pi:\Cov^{(\varphi)}(C/K) \to \Cov^{(\ell)}(E/K) \simeq \HH^1(K,E[\ell])\,.\]

\begin{Definition}
\label{DefCov0}
We define $\Cov_0^{(\varphi)}(C/K) \subset \Cov^{(\varphi)}(C/K)$ to be the subset consisting of elements $D$, such that $\Psi_\pi(D)$ is self orthogonal with respect to the Weil-pairing cup product
\[ \cup_\ell:\HH^1(K,E[\ell]) \times \HH^1(K,E[\ell]) \stackrel{\cup}{\To} \HH^2(K,E[\ell]\otimes E[\ell]) \stackrel{e_\ell}{\To} \HH^2(K,\mu_\ell) = \Br(K)[\ell]\,. \]
\end{Definition}

\begin{Remark}
Equivalently, $\Cov_0^{(\varphi)}(C/K)$ is the set of isomorphism classes of $\varphi$-coverings which map via $\Psi_\pi$ into the kernel of the obstruction map, $\operatorname{Ob}_\ell:\HH^1(K,E[\ell]) \to \Br(K)$, considered in \cite{CFOSS}. Indeed, it is known that $\operatorname{Ob}_\ell$ is quadratic and that the associated bilinear form is the cup product $\cup_\ell$ figuring in our definition.
\end{Remark}

If $k$ is a number field and $D \in \Cov^{(\varphi)}(C/k)$ is everywhere locally solvable, the local global principle for the Brauer group of $k$ shows that $D \in \Cov_0^{(\varphi)}(C/k)$. In other words, $\Sel^{(\varphi)}(C/k) \subset \Cov_0^{(\varphi)}(C/k)$.

%
%

We also have a geometric description of $\Cov_0^{(\varphi)}(C/K)$.

\begin{Lemma}
\label{Cov0property}
Let $D \in \Cov^{(\varphi)}(C/K)$. Then $D \in \Cov_0^{(\varphi)}(C/K)$ if and only if there is a model for $D$ as a genus one normal curve of degree $\ell$ in $\BP^{\ell-1}$ defined over $K$ with the property that the pull-back of any $x \in X \subset C$ is a hyperplane section.
\end{Lemma}

\begin{proof}
Fix isomorphisms $\psi_D:D \to E$ and $\psi_C:C \to E'$ (defined over $\bK$) such that the diagram \[ \xymatrix{ D \ar[d]_{\psi_D}\ar[r]^\rho& C \ar[d]_{\psi_C} \ar[r]^\pi& E \ar@{=}[d]\\
E\ar[r]_{\varphi}& E'\ar[r]_{\varphi'}& E }\] commutes. The $\ell$-covering $(D,\pi\circ\rho)$ is self orthogonal with respect to the Weil pairing cup product if and only if $\psi_D^*(\ell[0_E])$ is linearly equivalent to some $K$-rational divisor (see \cite{CFOSS},\cite{CreutzThesis}). On the other hand, $D$ admits a model as in the statement of the lemma if and only if $\rho^*[x]$ is linearly equivalent to some $K$-rational divisor, for each $x \in X$. It thus suffices to show that for all $x \in X$, $\psi_D^*(\ell[0_E])$ and $\rho^*[x]$ are linearly equivalent. For this we may work geometrically. The problem is then equivalent to showing that for any $\varphi'$-torsion point $P \in E'[\varphi']$, the pull-back of $P$ under $\varphi$ is linearly equivalent to $\ell[0_E]$. This follows from the well-known fact that two divisors on an elliptic curve are linearly equivalent if and only if they have the same degree and the same sum. Indeed, the divisors in question both have degree $\ell$ and sum to $0_E$ in the group $E(\bK)$.
\end{proof}

\section{The Descent Map}
\label{DescentMap}
In this section we define a map on $\Cov_0^{(\varphi)}(C/K)$ taking values in a quotient of the multiplicative group of a certain \'etale $K$-algebra. Ultimately we will see that this map is injective. Its image gives a concrete realization of $\Cov_0^{(\varphi)}(C/K)$ which is more amenable to computation.

\subsection{$G_K$-sets and \'etale algebras} Recall that $C \stackrel{\pi}\to E$ is a $\varphi'$-covering of $E$ and $X = \pi^{-1}(0_E)$ is a torsor under $E'[\varphi']$. Define $Y$ to be the set of divisors on $C$,
\[Y = \left\{\,(\ell-2)[x] + [x+P] + [x-P] \in \Div(\bC) \,:\, x \in X,\, P \in E'[\varphi']/\{\pm 1\} \,\right\}\,.\] Then $Y$ is a $G_K$-set of hyperplane sections of $C \subset \BP^{\ell-1}$ that are supported entirely on $X$. Use $F = \Map_K(X,\bK)$ and $H =\Map_K(Y,\bK)$ to denote the \'etale $K$-algebras corresponding to $X$ and $Y$. We will identify $K$ with the subalgebras of constant maps in $F$ and in $H$. The action of $G_K$ on $Y$ is determined by the action of $G_K$ on $X$. This induces a `norm map,'
\[ \partial:F = \Map_K(X,\bK) \ni \varphi \mapsto \left( (y=\sum_{x\in X} n_x[x]) \mapsto \prod_{x\in X}\varphi(x)^{n_x} \right) \in \Map_K(Y,\bK) = H\,. \] 

As a $G_K$-set, $Y$ splits as the union of (at least) the two $G_K$-stable subsets,
\begin{align*}
Y_1 &= \{ \ell[x] \,:\, x \in X \}\,,\\
Y_2 &= \{ (\ell-2)[x] + [x+P] + [x-P] \,:\, x \in X,\, 0 \ne P \in E'[\varphi']\}\,.
\end{align*}
This gives rise to a splitting of $H$ as $H \simeq H_1\times H_2 := \Map_K(Y_1,\bK)\times\Map_K(Y_2,\bK)$. As $G_K$-sets $Y_1$ and $X$ are isomorphic, so we may identify $H_1$ with $F$. The `norm map' $\partial$ also splits as $\partial = (\partial_1,\partial_2)$, where $\partial_1: F \ni \alpha \mapsto \alpha^\ell \in F \simeq H_1$ and $\partial_2 : F \to H_2$.

\begin{Remark}
For $\ell = 3$, $X$ is a set of three colinear points on $C \subset \BP^2$ and $Y_2$ is a singleton containing the line through these three points. The map $\partial_2:F \to H_2$ is the usual norm $N_{F/K}:F \to K$.
\end{Remark}

\subsection{The descent map}

Let $[\x] \in \Div(C \otimes_KF) = \Map_K(X,\Div(C))$ denote the map $x \mapsto [x]$. Similarly define $[\y] \in \Div(C \otimes_KH) = \Map_K(Y,\Div(C))$ to be $[\y]:y \mapsto y$. Choose a linear form $\mathfrak{l} \in H[u_1,\dots,u_\ell]$ cutting out the divisor $[\y]$. This means that $\frak{l}$ is a Galois equivariant family of linear forms parametrized by the $G_K$-set $Y$, with the property that evaluating the coefficients of $\frak{l}$ at $y \in Y$ yields a linear form $\frak{l}_{y} \in \bK[u_1,\dots,u_\ell]$ defining the hyperplane section $y \in \Div(\bC)$. Note that under the splitting, $H \simeq F\times H_2$, we have $\frak{l} = (\frak{l}_1,\frak{l_2})$ and the divisor defined by $\frak{l}_1$ is $\ell[\x]$.

\begin{Lemma}
Let $d = \sum_Pn_P[P] \in \Div(C)$ be any $K$-rational divisor on $C$ with support disjoint from $X$.
\begin{enumerate}
\item Evaluating $\frak{l}$ on $d$ gives a well defined element $\frak{l}(d) := \prod_{P} \frak{l}(P)^{n_P} \in H^\times/K^\times$.
\item $\frak{l}$ induces a unique homomorphism 
\[ \frak{l}:\Pic(C) \to \frac{H^\times}{K^\times\partial F^\times}\,.\] with the property that for all $d$ as above, the image of the class of $d$ is equal to the class of $\frak{l}(d)$.
\end{enumerate}
\end{Lemma}

\begin{Remark}
Results of this type are well known. This particular statement is a special case of \cite[Proposition 3.1]{CreutzThesis} (or \cite[Proposition 2.1]{Creutz2ndp}).
\end{Remark}

\begin{proof}
If $K'/K$ is any extension and $h \in \kappa(C\otimes_K K')^\times$ is a rational function, then the rule \[ d = \sum n_P P \mapsto h(d) = \prod h(P)^{n_P} \in K'^\times\,.\] defines a homomorphism from the group of $K'$-rational divisors with support disjoint from the support of $\diw(h)$ to $K'^\times$. Given $d$ as in the statement we can choose a linear form $u \in K[u_1,\dots,u_\ell]$ such that $d$ and the hyperplane section defined by $u$ have disjoint supports. Then $\frak{l}/u \in \kappa(C \otimes H)$ is a Galois equivariant family of rational functions. The corresponding homomorphisms patch together to give a homomorphism from the group of $K$-rational divisors with support disjoint from $X$ and the support of $\diw(u)$ to $H^\times$. Modulo the choice for $u$ we get a well defined homomorphism from $K$-rational divisors with support disjoint from $X$ to $H^\times/K^\times$. This proves the first statement.

For the second, define \[\frak{l}:\Pic(C) \to \frac{H^\times}{K^\times\partial(F^\times)}\] by setting the value of $\frak{l}$ on $\Xi \in \Pic(C)$ equal to the class of $\frak{l}(d)$, where $d \in \Div(C)$ is any $K$-rational divisor representing $\Xi$ with support disjoint from $X$. If this is well-defined, then it is clearly the unique homomorphism with the stated property. That such $d$ exists follows from \cite[page 166]{LangAV} where it is shown that any $K$-rational divisor class which is represented by a $K$-rational divisor contains a $K$-rational divisor avoiding a given finite set. 

Next we use Weil reciprocity to show that the result does not depend on the choice for $d$. Let $h \in \kappa(C)^\times$ be any rational function whose zeros and poles are disjoint from $X$ and choose a linear form $u \in K[u_1,\dots,u_\ell]$ whose corresponding divisor has support disjoint from $X$ and the support of $\diw(h)$. Weil reciprocity gives \[ \frac{\frak{l}}{u}(\diw(h)) = h\left(\diw\left(\frac{\frak{l}}{u}\right)\right) = \frac{h(\diw(\frak{l}))}{h(\diw(u))} \in H^\times \]

Define $\alpha \in \Map_K(X,\bK^\times) = F^\times$ by $\alpha:X \ni x \mapsto h(x) \in \bK^\times$. Now consider $\partial(\alpha) \in \Map_K(Y,\bK^\times) = H^\times$. The value of $\partial(\alpha)$ at $y = \sum n_x[x] \in Y$ is \[ \partial(\alpha)_y = \prod \alpha(x)^{n_x} = \prod h(x)^{n_x} = h(y) = h(\diw(\frak{l}_y))\,.\] This shows that $h(\frak{l}) = \partial(\alpha) \in \partial(F^\times)$. On the other hand, $h(\diw(u))$ clearly lies in $K^\times$. So $\frac{\frak{l}}{u}(\diw(h)) \in K^\times\partial F^\times$ which shows that the homomorphism is well defined.
\end{proof}

It is worth noting that for a point $P \in C(K)\setminus X$, the image of the class of  $[P]$ in $\Pic(C)$ under this homomorphism is given by evaluating $\frak{l}$ on any choice of homogeneous coordinates in $K$ for $P$. In general one must use the moving lemma to find a linearly equivalent divisor with support disjoint from $X$. We will abuse notation slightly by writing $\frak{l}(P)$ for the image of the class of $[P]$.

\begin{Proposition}
\label{DescentMapTheorem}
The choice of $\frak{l}$ induces a well defined map \[ \Phi:\Cov_0^{(\varphi)}(C/K) \To \frac{H^\times}{K^\times\partial F^\times}\,,\] with the property that if $(D,\rho)\in \Cov^{(\varphi)}_0(C/K)$ and $K\subset K'$ is any extension of fields with $Q \in D(K')$, then
\[ \Phi((D,\rho)) \equiv \frak{l}(\rho(Q)) \mod K'^\times \partial(F\otimes_K K')^{\times}\,. \]
\end{Proposition}

\begin{proof}
Let $(D,\rho) \in \Cov_0^{(\varphi)}(C/K)$. By Lemma \ref{Cov0property} we have a model for $(D,\rho)$ as a genus one normal curve of degree $\ell$ in $\BP^{\ell-1}= \BP^{\ell-1}(z_1:\dots:z_\ell)$ such that the pull-back of any $x \in X$ is a hyperplane section of $D$, defined by the vanishing of some linear form $h_x\in \bK[z_1,\dots,z_\ell]$. Moreover, these $h_x$ may be chosen so that they patch together to give a linear form $\frak{h} \in F[z_1,\dots,z_\ell]$ cutting out the divisor $\rho^*[\x]$ on $D\otimes_K F$. Since the zero divisor of $\frak{l}$ is $[\y] = \partial[\x] \in \Div(C\otimes_K H)$ we see that $\partial \frak{h}$ and $\frak{l}\circ\rho$ cut out the same divisor on $D$. Hence there exists some $\Delta \in H^\times$ such that
\begin{align}
\label{rel}
\frak{l} \circ \rho = \Delta\cdot \partial \frak{h}\,, 
\end{align} 
in the homogeneous coordinate ring of $D\otimes_K H$. We define $\Phi((D,\rho))$ to be the class of $\Delta$ in $H^\times/K^\times\partial F^\times$. Note that a different choice for $\frak{h}$ would alter $\Delta$ by an element of $K^\times\partial F^\times$.

Let us show that this does not depend on the model. Suppose $(D',\rho')$ is isomorphic to $(D,\rho)$. As above let $\frak{h} \in F[z_1,\dots,z_\ell]$ be a linear form cutting out the divisor $\rho^*[\x]$ on $D \otimes_KF$. By assumption we have an isomorphism of coverings $\psi : D' \rightarrow D$ defined over $K$ (i.e. such that $\rho'=\rho\circ\psi$). In the coordinate ring of $D'\otimes_K H$ we have
\begin{align}
\label{stringeq}
\Delta\cdot \partial(\frak{h}\circ\psi) = \psi^*(\Delta\partial \frak{h}) = \psi^*(\frak{l}\circ\rho) = \frak{l}\circ\rho\circ\psi = \frak{l}\circ\rho'\,.
\end{align}
The divisor on $D'$ cut out by $\frak{h}\circ \psi$ is $\rho'^*[\x]$, so the extremal terms in (\ref{stringeq}) define the image of $(D',\pi')$ under the descent map. Thus the image of $(D',\pi')$ is also the class of $\Delta$, which shows that $\Phi$ is well-defined.

It remains to show that $\Phi$ has the stated property. For this let $Q \in D(K')$. There exists a $K'$-rational divisor $d = \sum_in_iQ_i$ on $D$ linearly equivalent to $[Q]$ and such that the support of $d$ contains no points lying above $X$. The divisor $[\rho(Q)]$ on $C$ is linearly equivalent to the $K'$-rational divisor $\rho_*d := \sum_in_i[\rho(Q_i)]$ (e.g. \cite[II.3.6]{Silverman}). So $\frak{l}(\rho(Q))$ is represented by $\frak{l}(\rho_*d)$. On the other hand, the relation (\ref{rel}) defining $\Delta$ gives, $\frak{l}(\rho_*d) = \Delta\cdot \partial \frak{h}(d)$, since $\deg(d) = 1$. Now since $d$ is $K'$-rational, $\partial \frak{h}(d) \in K'^\times \partial (F\otimes_K K')^\times$. So $\frak{l}(\rho(Q))$ is represented by $\Delta$ as required.
\end{proof}

\subsection{Image of the descent map}

Recall that we identify $K$ with the constant maps in $F$ and $H$. Under this identification we have $\partial(a) = a^\ell$, for all $a \in K$. Thus $K^\times \subset \partial\bar{K}^\times \subset \partial\bar{F}^\times$. We define a subgroup
\begin{align}
\label{H0def}
\mathcal{H}_K^0 &= \frac{\left(\partial\bF^\times\right)^{G_K}}{K^\times\partial F^\times} \subset \frac{H^\times}{K^\times \partial F^\times}\,, 
\end{align}
and a subset
\begin{align}
\label{Hdef}
\mathcal{H}_K &= \frac{\left(\frak{l}(P)\cdot\partial\bF^\times\right)^{G_K}}{K^\times\partial F^\times}\,,
\end{align} where $P \in C(\bK)$ is any point.

The defining property of $\Phi$ and the next lemma show that the image of $\Phi$ is contained in $\mathcal{H}_K$. Ultimately we will see that $\mathcal{H}_K$ is equal to the image of $\Phi$.

\begin{Lemma}
$\mathcal{H}_K$ does not depend on the choice for $P \in C(\bK)$.
\end{Lemma}

\begin{proof}
Let $P' \in C(\bK)$ and choose some $(D,\rho) \in \Cov_0^{(\varphi)}(\bC/\bK)$. Fix a model for $D$ in $\BP^{\ell-1}$ as above and let $\frak{h}$ be a linear form with coefficients in $\bF$ defining the divisor $\rho^*[\x]$. Then, for some $\Delta \in \bH^\times$, we have $\frak{l}\circ\rho = \Delta \partial\frak{h}$ in the coordinate ring of $D \otimes_{\bK} \bar{H}$. If $Q,Q'$ are points of $D$ lying above $P$ and $P'$, we have
\[ \frac{\frak{l}(P)}{\frak{l}(P')} = \frac{\partial\frak{h}(Q)}{\partial\frak{h}(Q')} = \partial\left( \frac{\frak{h}(Q)}{\frak{h}(Q')}\right) \in \partial \bF^\times\,.\] This shows that the coset $\frak{l}(P)\cdot\partial\bF^\times$ does not depend on $P$. The same holds for its $G_K$-invariant subset, which proves the lemma.
\end{proof}

The next lemma shows that non membership in $\mathcal{H}_K$ is stable under base change.

\begin{Lemma}
\label{rescheat}
Suppose that $K'$ is an extension of $K$ and $\Delta \in H^\times$ is such that $\Delta \otimes_K 1$ represents a class in $\mathcal{H}_{K'}$. Then the class of $\Delta$ in $H^\times/K^\times\partial F^\times$ lies in $\mathcal{H}_K$.
\end{Lemma}

\begin{proof}
The assumptions imply that in $(\bH \otimes_{\bK}\bK')^\times$ we have
\[ \frac{\Delta}{\frak{l}(P)} \otimes 1 = \partial \alpha\,, \] for some $\alpha \in (\bF \otimes_{\bK}\bK')^\times \simeq \prod_{x\in X}\bK'^\times$ depending on $P \in C(\bK)$. For $x \in X$ we have $\Delta_x/\frak{l}_x(P) = \alpha_x^\ell$ in $\bK'$. This shows that $\alpha_x$ is algebraic over $\bK$, and hence lies in $\bK$. Thus $\alpha \in \bF$ and we have $\Delta = \frak{l}(P)\partial(\alpha)$ in $\bH$. This shows that $\Delta$ represents a class in $\mathcal{H}_K$.
\end{proof}

\section{Towards Cohomology}
\label{Cohomology}

\subsection{Affine maps}
We say that a map $f:X \to \bK^\times$ is {\em affine} if \[ \text{ for all } x \in X \text{ and } P,Q \in E'[\varphi'],\,\,\, f(x+P)\cdot f(x+Q) = f(x)\cdot f(x+P+Q)\,.\] An easy calculation shows  that this is actually equivalent to the a priori weaker condition \[ \text{ for all } x \in X \text{ and } P\in E'[\varphi'],\,\,\, f(x+P)\cdot f(x-P) = f(x)^2\,.\] This results in the following.
\begin{Lemma}
\label{partialcutsoutAff}
A map $\alpha \in \mu_\ell(\bF) = \Map(X,\mu_\ell)$ is affine if and only if $\partial_2(\alpha) = 1$.
\end{Lemma}

\begin{proof}
$\alpha$ lies in the kernel of $\partial_2$ if and only if for every $y = (\ell-2)[x] + [x+P] + [x-P] \in Y_2$ we have $\alpha(x)^{\ell-2}\alpha(x+P)\alpha(x-P) = 1$. 
\end{proof}

Let $\Aff(X,\mu_\ell)$ denote the $G_K$-module of affine maps from $X$ to $\mu_\ell$. Given an affine map $\alpha \in \Aff(X,\mu_\ell)$ and $x \in X$, projecting onto the linear part gives a well defined homomorphism $\left(P\mapsto \alpha(x+P)/\alpha(x)\right) \in \Hom(E'[\varphi'],\mu_\ell)$. Since $\alpha$ is affine this does not depend on the choice for $x \in X$, so we get a morphism of $G_K$-modules $\lambda:\Aff(X,\mu_\ell) \to \Hom(E'[\varphi'],\mu_\ell)$ whose kernel consists of the constant maps. Using the $\varphi$-Weil pairing we have an identification
\begin{align}
\label{WPident}
E[\varphi] \ni P \mapsto e_{\varphi}(P,\cdot) \in \Hom(E'[\varphi'],\mu_\ell)\,.
\end{align}
This gives an exact sequence of $G_K$-modules
\begin{align}
\label{projontolinear}
1 \to \mu_\ell \to \Aff(X,\mu_\ell) \stackrel{\lambda}\To \Hom(E'[\varphi'],\mu_\ell) = E[\varphi] \to 0\,.
\end{align}

Taking Galois cohomology of (\ref{projontolinear}) and identifying $\HH^2(K,\mu_\ell)$ with $\Br(K)[\ell]$, we obtain the following exact sequence
\begin{align}
\label{defUpsilon}
 \HH^1(K,\mu_\ell) \to \HH^1(K,\Aff(X,\mu_\ell)) \to \HH^1(K,E[\varphi]) \stackrel{\Upsilon}{\to} \Br(K)[\ell]\,.
 \end{align}
 
 \begin{Lemma}
 \label{kerUpsH0}
 There is an isomorphism  \begin{align}\label{kerUps} \Phi_0: \ker(\Upsilon) \simeq \frac{\left(\partial\bF^\times\right)^{G_K}}{K^\times\partial F^\times} = \mathcal{H}_K^0\,.\end{align} 
 \end{Lemma}
 
 \begin{proof}
Lemma \ref{partialcutsoutAff} gives a morphism of short exact sequences,
\begin{align}
\label{ComputeH1Aff} \xymatrix{ 
1 \ar[r] &\mu_\ell \ar[d]\ar[r] & \bK^\times \ar[d]\ar[r]^{\ell}&\bK^\times \ar[d]\ar[r] &1 \,\,\,\\ 
1 \ar[r] & \Aff(X,\mu_\ell)\ar[r]  &\bF^\times \ar[r]^{\partial} & \partial\bF^\times \ar[r] & 1 \,,}
\end{align} where the copies of $\bK^\times$ are embedded as the constant maps. Taking Galois cohomology gives a morphism of long exact sequences,
\[ \xymatrix{
1 \ar[r]& \mu_\ell(K) \ar[r]\ar[d]& K^\times \ar[r]^\ell\ar[d]& K^\times \ar[r]\ar[d] &\HH^1(K,\mu_\ell) \ar[r]\ar[d] & \HH^1(K,\bK^\times) = 0 \\
1 \ar[r]& \Aff_K(X,\mu_\ell) \ar[r]&F^\times \ar[r]^\partial & (\partial\bF^\times)^{G_K} \ar[r] &\HH^1(K,\Aff(X,\mu_\ell)) \ar[r] & \HH^1(K,\bF^\times) = 0, }\] where the equalities follow from Hilbert's Theorem 90. From this we see that $\mathcal{H}_K^0$ is isomorphic to the quotient of $\HH^1(K,\Aff(X,\mu_\ell))$ by the image of $\HH^1(K,\mu_\ell)$. On the other hand, the latter is isomorphic to the kernel of $\Upsilon$ by exactness of (\ref{defUpsilon}).
 \end{proof}

For any $x_0 \in X$ we can use the $\varphi$-Weil pairing to define a map \[E[\varphi]\times X \ni (P,x) \mapsto e_{\varphi}(P,x-x_0) \in \mu_\ell\,.\] A different choice of $x_0$ gives a map which differs by a constant. Nondegeneracy of the Weil pairing shows that we have a well defined embedding \begin{align}
\label{embedEphi}
E[\varphi] \ni P \mapsto e_{\varphi}(P,x-x_0) \in \mu_\ell(\bF)/\mu_\ell\,.
\end{align} 

The following lemma gives an alternative description of this embedding.

\begin{Lemma}
\label{WPlemma} Let $D \in \Cov_0^{(\varphi)}(C/\bK)$ ({\sc nb}: over $\bK$, not $K$) and let $\frak{h}$ denote a linear form (with coefficients in $\bF$) defining the pull-back of $[\x]$. For any $P \in E[\varphi]$, the image of $P$ under the composition $E[\varphi] \hookrightarrow \mu_\ell(\bF)/\mu_\ell \hookrightarrow \bF^\times/\bK^\times$ is equal to the class of $\frac{\frak{h}(P+Q)}{\frak{h}(Q)}$, where $Q \in D$ is any point chosen so that $\frak{h}(Q)$ and $\frak{h}(P+Q)$ are both defined and invertible in $\bF$.
\end{Lemma}

\begin{proof}
Fix isomorphisms $\psi_D:D \to E$ and $\psi:C \to E'$ (defined over $\bK$) such that the diagram \[ \xymatrix{ D \ar[d]_{\psi_D}\ar[r]^\rho& C \ar[d]_{\psi} \ar[r]^\pi& E \ar@{=}[d]\\
E\ar[r]_{\varphi}& E'\ar[r]_{\varphi'}& E }\] commutes. Let $x_0$ be the preimage of $0_{E'}$ under $\psi$, and let $Q_0$ be any preimage of $x_0$ under $\rho$. If $x \in X$, evaluating the coefficients of $\frak{h}$ at $x$ gives a linear form $h_x$ defining the pull-back of $[x]$ by $\rho$. Consider the function $h_x/h_{x_0} \in \kappa(\bar{D})^\times$ and its image $g_x = (\psi_D^{-1})^*(h_x/h_{x_0}) \in \kappa(\bar{E})^\times$. The divisor of $h_x/h_{x_0}$ is $\rho^*[x]-\rho^*[x_0]$, so by commutativity $\diw(g_x) = \varphi^*[(x-x_0)]-\varphi^*[0_{E'}]$. By definition of the $\varphi$-Weil pairing \cite[III.8, exer. 3.15]{Silverman}, for any $P \in E[\varphi]$,  \[ e_{\varphi}(P,x-x_0) = \frac{g_x(P + R)}{g_x(R)}\,,\] where $R \in E$ is any point chosen so that both numerator and denominator are defined and nonzero. Thus, we have 
\begin{align}
\label{eqqq}
e_{\varphi}(P,x-x_0) = \frac{h_x(P+\psi_D^{-1}(R))h_{x_0}(\psi_D^{-1}(R))}{h_x(\psi_D^{-1}(R))h_{x_0}(P+\psi_D^{-1}(R))}\,. \end{align} Considered as an element of $\bF^\times = \Map(X,\bK^\times)$ modulo the constant maps, the right-hand side of (\ref{eqqq}) is represented by the map \[ \frac{h(P+\psi_D^{-1}(R))}{h(\psi_D^{-1}(R))} = \left(x \mapsto \frac{h_x(P+\psi_D^{-1}(R))}{h_x(\psi_D^{-1}(R))}\right)\,.\] On the other hand, the left-hand side of (\ref{eqqq}) represents the image of $P$ in $\mu_\ell(\bF)/\mu_\ell$, so we are done.
\end{proof}

\subsection{Some diagrams}
The maps in (\ref{projontolinear}) and (\ref{embedEphi}) fit together to give a commutative and exact diagram:
\begin{align}
\label{diagram1}
\xymatrix{
& \mu_\ell \ar@{=}[r] \ar@{^{(}->}[d] & \mu_\ell \ar@{^{(}->}[d] & &\\
1 \ar[r] & \Aff(X,\mu_\ell) \ar[d]\ar[r] & \mu_\ell(\bF) \ar[d]\ar[r]^{\partial_2} & \partial_2(\mu_\ell(\bF))\ar@{=}[d] \ar[r] & 1 \\
1 \ar[r] & E[\varphi] \ar[d]\ar[r] & \frac{\mu_\ell(\bF)}{\mu_\ell} \ar[d]\ar[r]^{\partial_2} & \partial_2(\mu_\ell(\bF)) \ar[r] & 1\\
& 1 & 1 & &}
\end{align}

Taking Galois cohomology and making the identifications described below yields another commutative and exact diagram:
\begin{align}
\label{diagram2}
\xymatrix{
&& 1  \ar[d] & 1 \ar[d] & \\
0 \ar[r] & \kappa_K \ar@{=}[d] \ar[r] & \mathcal{H}_K^0 \ar[d]\ar[r]^{\pr_1} & \frac{F^\times}{K^\times F^{\times\ell}} \ar[d]\ar[r]^{\partial_{2,*}}& \HH^1(K,\partial_2\mu_\ell(\bF))\ar@{=}[d] \\
0 \ar[r]& \kappa_K \ar[r] & \HH^1(K,E[\varphi]) \ar[d]^{\Upsilon}\ar[r] & \HH^1(K,\frac{\mu_\ell(\bF)}{\mu_\ell}) \ar[d]\ar[r]^{\partial_{2,*}}& \HH^1(K,\partial_2\mu_\ell(\bF)) \\
&& \Br(K)[\ell] \ar@{=}[r] & \Br(K)[\ell] & }
\end{align}

where the map $\pr_1$ is induced by the projection $\pr_1:H^\times \simeq F^\times \times H_2^\times \to F^\times$ and $\kappa_K \simeq \frac{\HH^0\left(K,\partial_2\mu_\ell(\bF)\right)}{\partial_2\left(\HH^0\left(K,\frac{\mu_\ell\bF}{\mu_\ell}\right)\right)}$.

The bottom row of (\ref{diagram2}) is obtained directly by taking Galois cohomology of the bottom row of (\ref{diagram1}). Using Hilbert's Theorem 90 we may identify the quotient of $\HH^1(K,\mu_\ell(\bF))$ by $\HH^1(K,\mu_\ell)$ with $F^\times/K^\times F^{\times \ell}$. Using Lemma \ref{kerUpsH0} we identify the quotient of $\HH^1(K,\Aff(X,\mu_\ell))$ by $\HH^1(K,\mu_\ell)$ with $\mathcal{H}_K^0$. The top row of (\ref{diagram2}) is then seen to be the quotient of the long exact sequence coming from the top row of (\ref{diagram1}) by the images of $\HH^1(K,\mu_\ell)$. A completely formal diagram chase shows that the rows of (\ref{diagram2}) have isomorphic kernels (so that $\kappa_K$ is the kernel of the top row as well). 

Recall that the image of $\Phi$ is contained in $\mathcal{H}_K$ which is a coset of $\mathcal{H}_K^0$ (see (\ref{H0def}) and (\ref{Hdef})). We define $\mathcal{F}_K$ and $\mathcal{F}_K^0$ to be their images in $F^\times/K^\times F^{\times\ell}$ under $\pr_1$. By (\ref{diagram2}) we have that
\begin{align*}
\mathcal{F}_K^0 &= \left\{ \delta \in F^\times/K^\times F^{\times\ell}\,:\, \partial_{2,*}(\delta) = 0 \right\}\\ 
&\subset \left\{ \delta \in F^\times/K^\times F^{\times\ell}\,:\, \partial_2(\delta) \in H_2^{\times\ell} \right\}\\ 
&\subset \left\{ \delta \in F^\times/K^\times F^{\times\ell}\,:\, N_{F/K}(\delta) \in K^{\times\ell} \right\}\,.
\end{align*}

For $\ell=3$, it is relatively easy to show that all three sets are equal (Recall that in this case $\partial_2 = N_{F/K}$). Since $\mathcal{F}_K$ is a coset of $\mathcal{F}_K^0$, we find that there exist constants $c \in K^\times$ and $\beta \in H_2^\times$ such that
\begin{align}
\label{defFk}
\mathcal{F}_K &\subset \left\{ \delta \in F^\times/K^\times F^{\times\ell}\,:\, \partial_2(\delta) \in \beta H_2^{\times\ell} \right\}\\\notag
&\subset \left\{ \delta \in F^\times/K^\times F^{\times\ell}\,:\, N_{F/K}(\delta) \in cK^{\times\ell} \right\}\,.
\end{align}

These constants can be computed as follows. The linear form $\frak{l}_1$ cuts out the divisor $\ell[\x]$. Its norm evidently cuts out the divisor $\ell\sum_{x \in X}[x] = \ell\pi^*[0_E]$. The model for $C \subset \BP^{\ell-1}$ is given by the embedding corresponding to $\pi^*[0_E]$, so $\ell\pi^*[0_E]$ is $\ell$ times a hyperplane section. If this hyperplane is defined by, say $h \in K[u_1,\dots,u_\ell]$, then, for some $c \in K^\times$, we have a relation $N_{F/K}(\frak{l}_1) = ch^\ell$ in the homogeneous coordinate ring of $C$. Similarly $\partial_2([\x]) = [\y]$ from which it follows that there exists $\beta \in H_2^\times$ such that $\partial_2(\frak{l}_1) = \beta \frak{l}_2^\ell$ in the coordinate ring.

\subsection{Relation to the descent map} We now relate diagram (\ref{diagram2}) to the descent map of the previous section.

\begin{Lemma}
\label{kernelUps}
For any $\xi \in \HH^1(K,E[\varphi])$, 
$\Upsilon(\xi) = -\xi \cup_{\varphi} C$.
\end{Lemma}

\begin{Corollary}
\label{Cov0isphs}
$\Cov_0^{(\varphi)}(C/K)$ is either empty or is a principal homogeneous space for $\mathcal{H}_K^0$ under the action of twisting.
\end{Corollary}

\begin{Remark}
If $k$ is a number field and $C$ is everywhere locally solvable, then it is possible to show that $\Cov_0^{(\varphi)}(C/k) \ne \emptyset$. This follows from the fact that any element in $\HH^1(k,\Aff(X,\mu_\ell))$ which becomes trivial everywhere locally is itself trivial (c.f. \cite[Theorem 2.5]{CreutzThesis}).
\end{Remark}

\begin{proof}[Proof of Lemma \ref{kernelUps}]
Let $\xi \in \HH^1(K,E[\varphi])$, and let $\psi:C\to E'$ be an isomorphism (defined over some extension of $K$) such that $\varphi'\circ\psi = \pi$. The class of $-C$ in $\Cov^{(\varphi')}(E/K) = \HH^1(K,E'[\varphi'])$ is represented by the cocycle $G_K \ni \sigma \mapsto \psi - \psi^\sigma \in E'[\varphi']$. The cup product $-\xi \cup_{\varphi} C = -C \cup_{\varphi'} \xi$ is represented by the $2$-cocycle 
\[ G_K\times G_K \ni (\sigma,\tau) \mapsto e_{\varphi'}(\psi - \psi^\tau,\,\xi^\tau_\sigma) = e_{\varphi}(\xi_\sigma^\tau, \psi^\tau-\psi) \in \mu_\ell\,.\]

The value of the connecting homomorphism $\Upsilon$ on $\xi$ is defined by choosing a lift of $\xi$ under (\ref{projontolinear}) to a cochain with values in $\Aff(X,\mu_\ell)$, and then taking its coboundary. So $\Upsilon(\xi)$ is represented by the coboundary of the $1$-cochain $G_K \ni \sigma \mapsto e_{\varphi}(\xi_\sigma,\psi(\cdot))\in \Aff(X,\mu_\ell)$, which is the $2$-cocycle \[ \eta:G_K\times G_K \ni(\sigma,\tau) \mapsto \frac{ e_{\varphi}(\xi_\sigma,\psi)^\tau\cdot e_{\varphi}(\xi_\tau,\psi)}{e_{\varphi}(\xi_{\sigma\tau},\psi)} \in \mu_\ell\,.\] Using Galois equivariance of the Weil pairing and that $\xi$ is a $1$-cocycle, this simplifies to

\[ \eta(\sigma,\tau) = \frac{ e_\varphi(\xi_\sigma^\tau,\psi^\tau)}{e_{\varphi}(\xi_\sigma^\tau,\psi)} = e_{\varphi}(\xi_\sigma^\tau,\psi^\tau -\psi)\,. \]
This is same as the cup product computed above, which proves the lemma.
\end{proof}

\begin{proof}[Proof of corollary \ref{Cov0isphs}]
Recall that $\HH^1(K,E[\varphi])$ acts simply transitively on $\Cov^{(\varphi)}(C/K)$ by twisting. We need only show that this restricts to a transitive action of $\mathcal{H}_K^0\simeq\ker(\Upsilon)$ on $\Cov_0^{(\varphi)}(C/K)$. Let $D \in \Cov^{(\varphi)}(C/K)$ and $\xi \in \HH^1(K,E[\varphi])$. After identifying all with their images in $\Cov^{(\ell)}(E/K)$, the twist of $D$ by $\xi$ is the class of $D + \xi$. For the cup products defining $\Cov_0^{(\varphi)}(C/K)$ we have
\begin{align*}
(D +\xi) \cup_\ell (D + \xi) - D \cup_\ell D &= 2D \cup_\ell \xi + \xi \cup_\ell \xi\\
&= 2\xi \cup_{\varphi} \varphi_*D + \xi \cup_{\varphi} \varphi_*\xi \\
&= 2\xi \cup_{\varphi} C\,.
\end{align*}
The desired conclusion follows.
\end{proof}

\begin{Proposition}
\label{DescentMapIsAffine}
The descent map $\Phi$ is affine: if $(D,\rho) \in \Cov_0^{(\varphi)}(C/K)$, $\xi \in \ker(\Upsilon)$ and $(D_\xi,\rho_\xi)$ is the twist of $(D,\rho)$ by $\xi$, then \[ \Phi((D_\xi,\rho_\xi)) = \Phi((D,\rho))\cdot\Phi_0(\xi)\,,\] where $\Phi_0$ is the isomorphism in (\ref{kerUps}).
\end{Proposition}

\begin{Corollary}
\label{Injective}
The descent map $\Phi$ is injective. If $\Cov_0^{(\varphi)}(C/K)$ is nonempty, then its image is equal to $\mathcal{H}_K$ (which was defined in (\ref{Hdef})).
\end{Corollary}

\begin{proof}[Proof of Proposition \ref{DescentMapIsAffine}]
Let $\xi \in \ker(\Upsilon)$, and fix models for $(D,\rho)$ and $(D_\xi,\rho_\xi)$ as genus one normal curves of degree $\ell$ in $\BP^{\ell-1}$. We can also fix equations for an isomorphism (of coverings) $\psi:D_\xi \to D$ defined over $\bK$, with the property that $\psi^\sigma(Q) = \psi(Q)+\xi_\sigma$ for all $Q \in D_\xi$ and $\sigma \in G_K$.

Choose linear forms $\frak{h}$ and $\frak{h}_\xi$ with coefficients in $F$ defining the pull-backs of $[\x] \in \Div(C \otimes_KF)$ by $\rho$ and $\rho_\xi$, respectively. For some $\Delta,\Delta_\xi \in H^\times$, necessarily representing the images of $(D,\rho)$ and $(D_\xi,\rho_\xi)$ under $\Phi$, we have \[ \Delta\cdot\partial \frak{h} = \frak{l}\circ \rho \text{ and } \Delta_\xi\cdot\partial \frak{h}_\xi = \frak{l}\circ \rho_\xi\,\] in the coordinate rings of $D \otimes_KH$ and $D_\xi \otimes_KH$, respectively. Applying $\psi^*$ to the first relation and comparing with the second gives \[ \Delta\cdot\partial(\frak{h}\circ\psi) = \Delta_\xi\cdot\partial \frak{h}_\xi \] in the coordinate ring of $D_\xi$. Specializing to a point $Q$ in $D_\xi(\bK)$ not lying above any $x \in X \subset C$ (i.e. so that both $\frak{h}_\xi$ and $\frak{h}\circ\psi$ are invertible at $Q$) we have \[ \frac{\Delta_\xi}{\Delta} = \partial\Bigl(\frac{\frak{h}(\psi(Q))}{\frak{h}_\xi(Q)}\Bigr) \in (\partial{\bF^\times})^{G_K}\,.\] Note that $\frak{h}_\xi(Q)$ and $\frak{h}(\psi(Q))$ depend on a choice of homogeneous coordinates for $Q$, but that their ratio does not. This is $G_K$-invariant since $\Delta$ and $\Delta_\xi$ are in $H^\times$. Under the isomorphism $(\partial\bF^\times)^{G_K}/\partial F^\times \simeq \HH^1(K,\Aff(X,\mu_\ell))$ implicit in (\ref{kerUps}), $\partial\Bigl(\frac{\frak{h}(\psi(Q))}{\frak{h}_\xi(Q)}\Bigr)$ corresponds to the class of the $1$-cocycle \[ \eta:G_K \ni\sigma \mapsto \left(\frac{\frak{h}(\psi(Q))}{\frak{h}_\xi(Q)}\right)^\sigma\left(\frac{\frak{h}_\xi(Q)}{\frak{h}(\psi(Q))}\right) \in \mu_\ell(\bF) = \Map(X,\mu_\ell)\,,\] which a priori takes values in $\Aff(X,\mu_\ell) \subset \mu_\ell(\bF)$. We need to show that the image of this cocycle under the map induced by $\Aff(X,\mu_\ell) \to E[\varphi]$ is cohomologous to $\xi$. For this we make use of the following commutative diagram.

\begin{align}
\label{digram}
\xymatrix{ \Aff(X,\mu_\ell) \ar@{^{(}->}[r]\ar[d]& \mu_\ell(\bF) \ar@{^{(}->}[r]\ar[d]& \bF^\times \ar[d] \\
		E[\varphi] \ar@{^{(}->}[r]& \mu_\ell(\bF)/\mu_\ell \ar@{^{(}->}[r]& \bF^\times/\bK^\times }
\end{align}
Here the left square is the same as in diagram (\ref{diagram1}). Since the horizontal maps are all injective, it will be enough to show that, for any $\sigma \in G_K$, the images of $\xi_\sigma$ and $\eta_\sigma$ in the lower-right corner are equal.

Using the fact that $\frak{h}$ and $\frak{h}_\xi$ are defined over $H$ and rearranging, we have \[ \eta_\sigma = \left(\frac{\frak{h}(\psi^\sigma(Q^\sigma))}{\frak{h}(\psi(Q))}\right) \left(\frac{\frak{h}_\xi(Q)}{\frak{h}_\xi(Q^\sigma)}\right)\,.\] Making use of the fact that $\psi^\sigma(Q^\sigma) =\psi(Q^\sigma) + \xi_\sigma$ we can rewrite this as \[ \eta_\sigma = \left(\frac{\frak{h}(\psi(Q) + \xi_\sigma +(\psi(Q^\sigma)-\psi(Q))\,)}{\frak{h}(\psi(Q))}\right) \left(\frac{\frak{h}_\xi(Q^\sigma + (Q - Q^\sigma))}{\frak{h}_\xi(Q^\sigma)}\right)\,.\] By Lemma \ref{WPlemma} this represents the image of \[ \xi_\sigma + (\psi(Q^\sigma) - \psi(Q)) - (Q^\sigma-Q) \in E[\varphi]\] under the embedding given by the bottom row of (\ref{digram}). But \[(\psi(Q^\sigma)-\psi(Q)) - (Q^\sigma-Q) = 0_E\,,\] (see \cite[X.3.5]{Silverman}) so the images of $\eta_\sigma$ and $\xi_\sigma$ in the lower right corner of (\ref{digram}) are equal. From this the proposition follows.
\end{proof}

\begin{proof}[Proof of corollary \ref{Injective}]
Note that $\Cov_0^{(\varphi)}(C/K)$ is a principal homogeneous space for $\mathcal{H}_K^0$ by corollary \ref{Cov0isphs}. The fact that $\Phi$ is affine and that $\Phi_0$ is an isomorphism imply that $\Phi$ is injective and that its image is a coset of $\mathcal{H}_K^0$. $\mathcal{H}_K$ is a coset of $\mathcal{H}_K^0$ containing the image of $\Phi$, so the two must be equal.
\end{proof}

\begin{Proposition}
\label{DescentOnE}
The following diagram is commutative.
\[ \xymatrix{ \Pic^0(C) \ar[d]^{\frak{l}} \ar@{^{(}->}[r] & \Pic^0(\bC)^{G_K} \ar@{=}[rr]&& E'(K) \ar[d]^{\delta_{\varphi}} \\
\mathcal{H}^0_K \ar[rr]^{\Phi_0^{-1}} && \ker(\Upsilon) \ar@{^{(}->}[r] & \HH^1(K,E[\varphi]) } \] Here $\delta_{\varphi}$ is the connecting homomorphism in the Kummer sequence.
\end{Proposition}

\begin{proof}[Proof of Proposition \ref{DescentOnE}]
Let $P \in \Pic^0(C) \subset E'(K)$ and choose a representative $d \in \Div(C)$ whose support is disjoint from $X$. Write $d$ as a difference $d=d_1 - d_2$ of effective divisors and write each $d_i$ as a sum $d_i = \sum_{j=1}^n Q_{i,j}$ of $n = \deg(d_1)=\deg(d_2)$ (possibly non-distinct) points on $C$. Now choose any $(D,\rho) \in \Cov_0^{(\varphi)}(C/\bK)$ and a linear form $\frak{h}$ with coefficients in $\bF$ defining the pull-back of $[\x]$. For each $Q_{i,j}$ in the support of $d$, choose a point $R_{i,j} \in D$ such that $\rho(R_{i,j})=Q_{i,j}$. These choices are such that, as points on $E'$,
\[ \varphi\bigl( R_{i,j}-R_{i',j'} \bigr) = \bigl( Q_{i,j}-Q_{i',j'}\bigr)\,,\] for any $i,j,i',j'$. In particular, $\varphi\left(\sum_{j=1}^n (R_{1,j} - R_{2,j})\right) = d$. So the image of $P$ under the connecting homomorphism is given by the cocycle
\[ \sigma \mapsto \left(\sum_{j=1}^n (R_{1,j}^\sigma - R_{2,j}^\sigma) - \sum_{j=1}^n (R_{1,j} - R_{2,j})\right) \in E[\varphi]\,.\]

On the other hand, the image of $P$ under $\frak{l}$ is represented by 
\[\frac{\frak{l}(d_1)}{\frak{l}(d_2)} =\prod_{j=1}^n\frac{\frak{l}(Q_{1,j})}{\frak{l}(Q_{2,j})} \in H^\times\,.\] The condition defining the image of $(D,\rho)$ under the descent map is, in the coordinate ring of $D \otimes_KH$, $\frak{l}\circ\rho = \Phi((D,\rho))\cdot \partial \frak{h}$. Since $\deg(d)=0$, this gives \[ \prod_{j=1}^n\frac{\frak{l}(Q_{1,j})}{\frak{l}(Q_{2,j})} = \prod_{j=1}^n \frac{\partial \frak{h}(R_{1,j})}{\partial \frak{h}(R_{2,j})} =  \partial\left(\prod_{j=1}^n \frac{\frak{h}(R_{1,j})}{\frak{h}(R_{2,j})}\right)  \in \partial \bF^\times\,.\]

Under the isomorphism $(\partial\bF^\times)^{G_K}/K^\times\partial F^\times \simeq \HH^1(K,\Aff(X,\mu_\ell))/K^\times$, $\frak{l}(P)$ is sent to the class of the cocycle $\sigma \mapsto \alpha^\sigma/\alpha\,,$ where $\alpha \in \bF^\times$ is any element such that $\partial \alpha$ represents $\frak{l}(P)$. The argument above shows we may take $\alpha = \left(\prod_{j=1}^n \frac{\frak{h}(R_{1,j})}{\frak{h}(R_{2,j})}\right)$. Hence, the image of $P$ in $\HH^1(K,\Aff(X,\mu_\ell))/K^\times$ is represented by the cocycle $\eta$ sending $\sigma \in G_K$ to 
\begin{align*}
\eta_\sigma &= \left(\prod_{j=1}^n \frac{\frak{h}(R_{1,j})}{\frak{h}(R_{2,j})}\right)^\sigma\cdot \left(\prod_{j=1}^n \frac{\frak{h}(R_{2,j})}{\frak{h}(R_{1,j})}\right)\\ &= \left(\prod_{j=1}^n \frac{\frak{h}^\sigma(R_{1,j}^\sigma)}{\frak{h}^\sigma(R_{2,j}^\sigma)}\right)\cdot \left(\prod_{j=1}^n \frac{\frak{h}(R_{2,j})}{\frak{h}(R_{1,j})}\right)\\
&= \left(\prod_{j=1}^n \frac{\frak{h}^\sigma(R_{2,j}^\sigma +(R_{1,j}^\sigma-R_{2,j}^\sigma))}{\frak{h}^\sigma(R_{2,j}^\sigma)}\right)\cdot \left(\prod_{j=1}^n \frac{\frak{h}(R_{1,j}+(R_{2,j}-R_{1,j}))}{\frak{h}(R_{1,j})}\right)\,.
\end{align*}
Note that both $\frak{h}$ and $\frak{h}^\sigma$ are linear forms satisfying the hypothesis of Lemma \ref{WPlemma}. Applying this lemma as was done in the proof of proposition \ref{DescentMapIsAffine}, to each factor appearing, we see that the image of $\eta_\sigma$ in $\HH^1(K,E[\varphi])$ is equal to the class of the cocycle \[ G_K\ni \sigma \mapsto \sum_{j=1}^n\left((R_{1,j}^\sigma-R_{2,j}^\sigma) - ((R_{1,j}-R_{2,j})\right) \in E[\varphi]\,.\] This is the same as the image under the connecting homomorphism, so the diagram commutes.
\end{proof}

\section{Inverse of the descent map}
\label{Geometry}
The purpose of this section is to prove the following theorem.

\begin{Theorem}
\label{InversePhi}
Given $\Delta \in \mathcal{H}_K$ we can explicitly compute a set of $\ell(\ell-3)/2$ quadrics (when $\ell=3$, a ternary cubic) with coefficients in $K$ which define a genus one normal curve $D_\Delta$ of degree $\ell$ in $\BP^{\ell-1}$ and a tuple $\rho_\Delta$ of homogeneous forms of degree $\ell$ which define a morphism $D_\Delta \to C \subset \BP^{\ell-1}$ giving $D_\Delta$ the structure of a $\varphi$-covering of $C$. Moreover, the class of $(D_\Delta,\rho_\Delta)$ lies in $\Cov_0^{(\varphi)}(C/K)$ and its image under $\Phi$ is $\Delta$.
\end{Theorem}

The construction gives an explicit inverse to the descent map, and hence the following corollary. This removes the assumption $\Cov_0^{(\varphi)}(C/K) \ne \emptyset$ in corollary \ref{Injective}.

\begin{Corollary}
\label{inversePhi2}
The descent map $\Phi$ gives an isomorphism of affine spaces $\Phi:\Cov_0^{(\varphi)}(C/K) \simeq \mathcal{H}_K$.
\end{Corollary}

Let $\BP_X$ be the projective space over $K$ whose coordinates correspond to the elements of $X$. We have an identification of $(F\setminus\{0\})/K^\times$ and $\BP_X(K)$. The action of $F^\times$ on $(F\setminus\{0\})/K^\times$ by multiplication corresponds to an action of $F^\times$ on $\BP_X$ by $K$-automorphisms.

Given $\Delta \in H^\times$, we can define a scheme $\tilde{D}_\Delta \subset \BP_X \times C$ by the rule
\[ \left(z,(u_1:\dots:u_\ell)\right) \in \tilde{D}_\Delta \Leftrightarrow \exists\,a \in K^\times \text{ such that } \frak{l}(u_1,\dots,u_\ell) = a\Delta\partial(z)\,.\]

This condition is invariant under scaling and the action of the Galois group, so $\tilde{D}_\Delta$ is defined over $K$. Moreover, the action of $\alpha \in F^\times$ on $\BP_X$ induces a $K$-isomorphism $\tilde{D}_{\partial(\alpha)\Delta} \simeq \tilde{D}_{\Delta}$. Thus $\tilde{D}_\Delta$ only depends on the class of $\Delta$ in $H^\times/K^\times\partial F^\times$. Define $D_\Delta \subset \BP_X$ to be the image of $\tilde{D}_\Delta$ under the projection of $\BP_X \times C$ onto the first factor. Then $\tilde{D}_\Delta$ is the graph of a morphism $\rho_\Delta:D_\Delta \to C$.

The following lemma gives some justification for this construction.
\begin{Lemma}
\label{imisright}
If $(D_\Delta,\rho_\Delta)$ is a $\varphi$-covering of $C$, then its class lies in $\Cov_0^{(\varphi)}(C/K)$ and its image under $\Phi$ is represented by $\Delta$.
\end{Lemma}

\begin{proof}
For the first statement it is enough to show that the pull back of any $x \in X$ is a hyperplane section of $D_\Delta$ (see Lemma \ref{Cov0property}). For this we can work geometrically. Over $\bK$, $\frak{l}_1$ splits as $(\frak{l}_x)_{x \in X}$ where $\frak{l}_x$ defines the hyperosculating plane to $C$ at $x$. The condition defining $\tilde{D}_\Delta$ gives $\frak{l}_x = a\Delta_xz_x^\ell$, for $x \in X$. From this it is clear that the fiber above $x \in X$ is cut out by the hyperplane $z_x = 0$. This proves the first statement. It also shows that $z$ is a linear form defining the pullback of $[\x]$ under $\rho_\Delta$. From the defining equation we see that $\frak{l}\circ \rho_\Delta = a\Delta \partial(z)$, from which it follows that $\Phi((D_\Delta,\pi_\Delta))$ is represented by $\Delta$.
\end{proof}

To make the construction more explicit, we can proceed as follows. Suppose $\Delta = (\Delta_1,\Delta_2) \in F^\times\times H_2^\times \simeq H^\times$. The equation $\frak{l}(u_1,\dots,u_\ell) = \Delta a \partial(z)$, where $a \in K^\times$, $z \in F^\times$ splits as a pair of equations
\begin{align}\label{startingeqs} \frak{l}_1(u_1,\dots,u_\ell) = \Delta_1az^\ell && \frak{l}_2(u_1,\dots,u_\ell) = \Delta_2a\partial_2(z)\,,\end{align} over $F$ and $H_2$, respectively. In terms of a basis for $F$ over $K$, $z^\ell$ and $\partial_2(z)$ can be written as forms of degree $\ell$ in $K[z_1,\dots,z_\ell]$.

When $\ell = 3$, writing $\frak{l}_1(u_1,\dots,u_\ell) = \Delta_1 az^\ell$ in terms of this basis and comparing coefficients yields $3$ equations, linear in the $u_i$, cubic in $z_j$ and with coefficients in $K$. Together with $\frak{l}_2 = a\Delta_2\partial_2(z)$ we have $4$ equations of this form (recall $H_2 = K$ when $\ell = 3$). Using linear algebra this system of equations reduces to
\begin{align*}
u_1 &= \rho_1(z_1,z_2,z_3),\\
u_2 &= \rho_2(z_1,z_2,z_3),\\
u_3 &= \rho_3(z_1,z_2,z_3),\\
0 &= f(z_1,z_2,z_3),
\end{align*}
where all forms on the right hand side are of degree $3$. Then $D_\Delta$ is the plane cubic with homogeneous ideal generated by $f$ and the $\rho_i$ define a map to $\BP^2$ (which contains $C$).

When $\ell \ge 5$ the same approach will yield homogeneous forms of degree $\ell$. This is perfectly acceptable for defining $\rho_\Delta$, however the model for $D_\Delta$ should be defined as an intersection of quadrics (see \cite[I.1.3]{CFOSS}). Since $F$ is a subalgebra of $H_2$ we may consider both equations in (\ref{startingeqs}) as being defined over $H_2$. There is a quadratic form $Q(z)$ with coefficients in $H_2$ such that $\partial_2(z) = z^{\ell-2}Q(z)$ (this is clear from the definition of $Y_2$ and $\partial_2$). To get something homogeneous we take the ratio of the two equations  in (\ref{startingeqs}) and multiply through by $z^2$. This gives
\begin{align}
\label{homogeq}
\frac{\frak{l}_2(u)}{\frak{l}_1(u)} z^2 = \frac{\Delta_2}{\Delta_1}Q(z)\,. \end{align} Writing this out in a basis for $H_2$ over $K$ (extending that for $F$ over $K$ used above) yields $\ell(\ell-1)/2$ quadrics in $z$ whose coefficients are $K$-rational rational functions on $C$. Elimination will result in some number of quadrics in $z$ with coefficients in $K$. Then $D_\Delta$ is the subscheme of $\BP_X \simeq \BP^{\ell-1}$ whose homogeneous ideal is generated by these quadrics, and $\rho_\Delta$ is obtained as above.

\begin{Remark}
From this explicit construction it is not a priori clear that $\rho_\Delta$ maps $D_\Delta$ to $C$. That this is so will become evident in the proof of Theorem \ref{InversePhi} below.
\end{Remark}

\begin{Lemma}
Elimination of $u_1,\dots,u_\ell$ from the $\ell(\ell-1)/2$ quadrics above results in $\ell(\ell-3)/2$ linearly independent quadrics with coefficients in $K$.
\end{Lemma}

\begin{proof}
To prove this we may work geometrically. Over $\bK$ the linear form $\frak{l}$ splits as $\frak{l} = (\frak{l}_{(x,P)})$, where $x \in X$, $P \in E'[\varphi']/\{\pm1\}$ and $\frak{l}_{(x,P)}$ cuts out the divisor $(\ell-2)[x] + [x+P] + [x-P]$ on $C$. For distinct $(x,P)$ we have distinct rational functions \[ G_{(x,P)} = \frac{\frak{l}_{(x,P)}}{\frak{l}_{(x,0)}} \in \kappa(\bC)^\times\,, \] with divisors $[x+P]+[x-P]-2[x]$. In particular, for any $x$, the functions $G_{(x,P)}$ lie in the Riemann-Roch space $\mathcal{L}(2[x])$, which has dimension $2$. If we fix $P_0 \in E'[\varphi']\setminus\{0\}$, then for any $P\in \frac{E'[\varphi']\setminus\{0\}}{\{\pm 1\}}$, we can find $a_{P},b_{P} \in \bK$ such that \[ G_{(x,P)} = a_{P}G_{(x,P_0)} + b_{P}\,.\] The functions $G_{(x,P)}$ have distinct divisors, so $a_P,b_P \ne 0$, for $P \ne P_0$.

In an appropriate basis for $\bF$ over $\bK$, the homogeneous equation (\ref{homogeq}) corresponds to a system of equations \[ G_{(x,P)}\cdot z_x^2 = \Delta_{(x,P)}/\Delta_{(x,0)}\cdot z_{x+P}z_{x-P}\,, \] parametrized by $(x,P) \in X \times \frac{E'[\varphi']\setminus \{0\}}{\{\pm 1\}}$. Using the relations above to eliminate the $G_{(x,P)}$ we obtain a set of quadrics
\[ a_{P}\frac{\Delta_{(x,P_0)}}{\Delta_{(x,0)}}\cdot z_{x+P_0}z_{x-P_0} + b_{P}\cdot z_x^2 =\frac{\Delta_{(x,P)}}{\Delta_{(x,0)}}\cdot z_{x+P}z_{x-P}\,,\] parametrized by $P \in \frac{E'[\varphi'] \setminus\{0,\pm P_0\}}{\{\pm1\}}$ and with coefficients in $\bK$. The coefficients here are all nonzero elements of $\bK$, so the quadrics are linearly independent. Thus, we have a set of $\#X\cdot\#\left(\frac{E'[\varphi'] \setminus\{0,\pm P_0\}}{\{\pm1\}}\right) = \ell(\ell-3)/2$ independent quadrics as required.
\end{proof}

It remains to prove that $(D_\Delta,\rho_\Delta)$ is in fact a $\varphi$-covering of $C$.

\begin{proof}[Proof of Theorem \ref{InversePhi}]
Fix an isomorphism (over $\bK$) of $C$ and $E'$ and use it to identify the two. We may arrange for this isomorphism to identify $X$ and $E'[\varphi']$. Let us compute the image under the descent map of the $\varphi$-covering in $\Cov_0^{(\varphi)}(C/\bK)$ given by $E \stackrel{\varphi}\to E' = C$. For this we should embed $E$ in $\BP^{\ell-1}$ in such a way that the pull back of any flex point is a hyperplane section. This amounts to finding a basis for the Riemann-Roch space of the divisor $\varphi^*[0_{E'}]$, which has dimension $\ell = \#E'[\varphi']$. For each $x \in E'[\varphi']$ choose a function $G_x \in \kappa(\bar{E})$ with $\diw(G_x) = \varphi^*[x] - \varphi^*[0_{E'}]$. The standard construction of the $\varphi$-Weil pairing shows that such functions exist. Nondegeneracy of the pairing shows that the $G_x$ are linearly independent, and hence form a basis for $\mathcal{L}(\varphi^*[0_{E'}])$.

This gives an embedding $g: E \ni Q \mapsto (G_x(Q))_{x \in E'[\varphi']} \in \BP_X$. It is evident that the pull back of any $x \in E'[\varphi']$ is the hyperplane section of $g(E)$ cut out by $z_x = 0$. Let $Q \in E\setminus E[\ell]$. Then $\varphi(Q) \notin E'[\varphi']$, so $\frak{l}(\varphi(Q))$ is invertible in $\bH$. The image of $(E,\varphi)$ under the descent map is represented by the $\bar{\Delta} \in \bH^\times$ such that $\frak{l}(\varphi(Q)) = \bar{\Delta}\partial(g(Q))$. This implies that $\bar{\Delta}$ represents a class in $\mathcal{H}_{\bK}$.

Let $D_{\bar{\Delta}} \stackrel{\rho_{\bar{\Delta}}}\To \BP^{\ell-1}$ be as defined by the construction above. It is clear that $\rho_{\bar{\Delta}} \circ g = \varphi$ on $E \setminus E[\ell]$, and that the image of this open subscheme under $g$ is contained in $D_{\bar{\Delta}}$. Since $D_{\bar{\Delta}}$ is complete, this is then true on all of $E$. Thus $g(E) \subset D_{\bar{\Delta}}$ and $\rho_{\bar{\Delta}} \circ g = \varphi$. By definition $g(E)$ is a genus one normal curve of degree $\ell$. Its homogeneous coordinate ring is generated by a $\bK$-vector space of quadrics of dimension $\ell(\ell-3)/2$ (resp. a ternary cubic for $\ell=3$). We already have such collection of quadrics (resp. a ternary cubic) which vanish on $D_{\bar{\Delta}}$, so $g(E) = D_{\bar{\Delta}}$. This proves the theorem for $\bar{\Delta}$.

Now let $\Delta \in H^\times$ be a representative for some class in $\mathcal{H}_K$, and let $D_\Delta \stackrel{\rho_\Delta}{\To} \BP^{\ell-1}$ be as given by the construction. Since $\bar{\Delta}$ represents a class in $\mathcal{H}_{\bK}$, the ratio $\Delta/\bar{\Delta}$ lies in $\partial\bF^\times$. Therefore $D_\Delta$ and $D_{\bar{\Delta}}$ are $\bK$-isomorphic as a $\BP^{\ell-1}$-schemes. It follows that $\rho_\Delta$ gives $D_\Delta$ the structure of a $\varphi$-covering of $C \subset \BP^{\ell-1}$. The theorem then follows from Lemma \ref{imisright}.
\end{proof}

\section{Computing the Selmer Set}
\label{ComputingSel}
We now specialize to the case that $K = k$ is a number field. The material of the preceding sections can be applied both to $k$ and to any completion $k_v$. To objects defined over $k$ we attach subscripts $v$ to denote the corresponding object over $k_v$ obtained by extension of scalars. We assume $C \in \Sel^{(\varphi')}(E/k)$ and is embedded in $\BP^{\ell-1}$ using the linear system corresponding to the pull back of $[0_E]$ under the covering map. Since $C$ is everywhere locally solvable, the natural map $\Pic(C) \to \Pic(\bC)^{G_k}$ is an isomorphism. We assume that the constants $c \in k^\times$ and $\beta \in H_2^\times$ (defined by (\ref{defFk})) are integral and that all coefficients involved in the linear form $\frak{l}$ are likewise integral. This can be achieved by scaling.

Functoriality of $\frak{l}$ gives rise to the following commutative diagram. \[ \xymatrix{ \Pic(C) \ar[rr]^{\frak{l}} \ar[d] && \frac{H^\times}{k^\times\partial F^\times} \ar[d]^{\prod\res_v} \\ \prod \Pic (C \otimes_kk_v) \ar[rr]^{\prod\frak{l}_v} && \prod \frac{H_v^\times}{k_v^\times\partial F_v^\times}} \] We make identifications $\Pic^1(C\otimes_kk_v) = C(k_v)$ and $\Pic^0(C\otimes_kk_v) = E'(k_v)$.

\begin{Definition}
We define the {\em algebraic $\varphi$-Selmer set of $C$} to be
\[ \Sel_{alg}^{(\varphi)}(C/k) = \left\{ \Delta \in \mathcal{H}_k \,:\, \text{ for all primes $v$, } \res_v(\Delta) \in \frak{l}(C(k_v)) \right\}\,.\] 
We define the {\em algebraic $\varphi$-Selmer group of $E'$} to be
\[ \Sel_{alg}^{(\varphi)}(E'/k) = \left\{ \Delta \in \mathcal{H}^0_k \,:\, \text{ for all primes $v$, } \res_v(\Delta) \in \frak{l}(E'(k_v)) \right\}\,.\] 
\end{Definition}

Proposition \ref{DescentOnE} shows that $\Sel_{alg}^{(\varphi)}(E'/k)$ is isomorphic to the $\varphi$-Selmer group of $E'$. The $\varphi$-Selmer set of $C$ is an affine space for the $\varphi$-Selmer group of $E'$. It is also evident that $\Sel_{alg}^{(\varphi)}(C/k)$ is an affine space over $\Sel_{alg}^{(\varphi)}(E'/k)$ (i.e. a coset inside $H^\times/k^\times\partial F^\times$). Propositions \ref{DescentMapTheorem}, \ref{DescentMapIsAffine} and Corollary \ref{Injective} show that the descent map gives an isomorphism of affine spaces \[\Phi:\Sel^{(\varphi)}(C/k) \To \Sel_{alg}^{(\varphi)}(C/k)\,.\]

To perform a $\varphi$-descent on $C$ it thus suffices to compute the algebraic $\varphi$-Selmer set. Using the method of section \ref{Geometry}, one can then construct explicit models for the elements of the $\varphi$-Selmer set as genus one normal curves of degree $\ell$ in $\BP^{\ell-1}$. An algorithm for computing the algebraic Selmer set is given in Theorem \ref{ComputeSel} below. For its statement we require the following notation.

Let $F'$ denote the splitting field of $X$. Over this field $\frak{l}_1$ splits as $(\frak{l}_x)_{x \in X}$. We say that $\frak{l}_1$ has bad reduction at a prime $v$ of $k$ if there a prime $w$ of $F'$ above $v$ and some $x \in X$ such that $\frak{l}_x \equiv 0 \mod w$.

Let $S$ be the finite set of primes of $k$ consisting of those primes such that 
\begin{enumerate}
\item $v \mid \ell \cdot c$, or
\item $C$ has bad reduction at $v$, or
\item $\frak{l}_1$ has bad reduction at $v$, or
\item $v$ ramifies in $F$.
\end{enumerate}

\begin{Theorem}
\label{ComputeSel}
The following algorithm returns a set of representatives in $H^\times$ for the algebraic $\varphi$-Selmer set of $C$.

{\sc Compute $\Sel_{alg}^{(\varphi)}(C/k)$:}
\begin{enumerate}
\item Compute a finite subset $V_1 \subset F^\times$ of representatives for the subgroup of $F^\times/k^\times F^{\times \ell}$ which is unramified outside $S$.
\item Compute the finite set 
\[ V_2 := \left\{ (\delta,\varepsilon) \in V_1\times H_2^\times\,:\, \partial_2(\delta) = \beta\varepsilon^\ell\,\right\}\,.\]
\item For each prime $v \in S$ compute
\[ G_v := \frak{l}_v(C(k_v)) \subset \left\{ (\delta,\varepsilon) \in F_v^\times\times H_v^\times\,:\, \partial_2(\delta) = \beta\varepsilon^\ell\,\right\}/k_v^\times\partial F_v^\times\,.\]
\item Return the set 
\[ V_3 := \left\{ (\delta,\eps) \in V_2 \,:\, \res_v(\delta,\eps) \in G_v, \text{ for all $v \in S$ } \right\}\,.\]
\end{enumerate}
\end{Theorem}

For the most part the steps in this algorithm are typical of explicit descents. The first step can be achieved by computing certain $S$-unit and class group information in $F$ (see for example \cite[Proposition 12.8]{PoonenSchaefer}). The second step requires only extracting $\ell$-th roots in $H_2^\times$. By (\ref{defFk}) $\mathcal{H}_{k_v}$ is contained in the set 
\[ \left\{ (\delta,\varepsilon) \in F_v^\times\times H_v^\times\,:\, \partial_2(\delta) = \beta\varepsilon^\ell\,\right\}/k_v^\times\partial F_v^\times\,. \] 
By Hensel's lemma this set is finite and $\frak{l}_v:C(k_v) \to \mathcal{H}_{k_v}$ is locally constant. The sizes of the local images in step (3) can be determined a priori. To compute $G_v$ in practice we compute the images of random $k_v$-points (given up to sufficient precision) until their images generate a sufficiently large space. For further details we refer the reader to the discussion of the analogous situations considered in \cite{SchaeferStoll,CreutzThesis}. The fourth step can be reduced to linear algebra over $\F_\ell$.

The proof of Theorem \ref{ComputeSel} will make use of the next few lemmas.

\begin{Lemma}
\label{unrKernel}
Suppose $v$ is a prime of $k$ that does not divide $\ell$ and does not ramify in $F$. Then an element of $\mathcal{H}_{k_v}^0$ is unramified if and only if its image in $\mathcal{F}_{k_v}^0$ is unramified.
\end{Lemma}
\begin{proof}
Let $k_v^{nr}$ be the maximal unramified extension of $k_v$. Recall that the kernel of the map $\pr_1:\mathcal{H}_{k_v^{nr}}^0 \to \mathcal{F}_{k_v^{nr}}^0$ is denoted by $\kappa_{k_v^{nr}}$. If $v$ is unramified in $F$ and does not divide $\ell$, then all points of $X$ are defined over $k_v^{nr}$. This follows from the criterion of N\'eron-Ogg-Shafarevich. We claim that $\kappa_{k_v^{nr}}$ is then trivial. Indeed, since the action of the inertia group on $X$ is trivial, $\kappa_{k_v^{nr}}$ reduces to the quotient of $\partial_2\mu_\ell(\bF_v)$ by $\partial_2(\mu_\ell(\bF_v)/\mu_\ell)$. But $\mu_\ell(\bF_v)$ and $\mu_\ell(\bF_v)/\mu_\ell$ have the same image since $\mu_\ell \subset \ker\partial_2$.

Now consider the following diagram with exact rows
\[\xymatrix{ 0 \ar[r]& \kappa_{k_v} \ar[d] \ar[r]& \mathcal{H}_{k_v}^0 \ar[r]\ar[d]& \mathcal{F}_{k_v}^0 \ar[r]\ar[d] & 0\,\,\, \\
0 \ar[r]& \kappa_{k_v^{nr}}  \ar[r]& \mathcal{H}_{k_v^{nr}}^0 \ar[r]& \mathcal{F}_{k_v^{nr}}^0 \ar[r] & 0\,. }\] By definition, the unramified subgroups are the kernels of these vertical maps. So it will suffice to show that $\mathcal{H}_{k_v^{nr}}^0\rightarrow \mathcal{F}_{k_v^{nr}}^0$ is injective. This follows from exactness since, as we have just seen, the lower-left term is trivial.
\end{proof}

\begin{Lemma}
\label{Pic0unr}
If $v \nmid \ell$, and $\ell$ does not divide the product of the Tamagawa numbers of $E$ and $E'$ at $v$, then $\frak{l}_v(E'(k_v)) \subset \mathcal{H}^0_{k_v} \subset \HH^1(k_v,E[\varphi])$ is equal to the unramified subgroup, \[\HH^1_{nr}(k,E[\varphi]) := \ker\left( \HH^1(k_v,E[\varphi]) \to \HH^1(k_v^{nr},E[\varphi]) \right)\,.\]
\end{Lemma}

\begin{proof}
Proposition \ref{DescentOnE} shows that $\frak{l}_v(E'(k_v))$ equals the image of the connecting homomorphism $E'(k_v) \to \HH^1(k_v,E[\varphi])$, which is itself equal to the unramified subgroup by \cite[Lemma 3.1]{Schaefer} and \cite[Lemma 3.1]{SchaeferStoll}.
\end{proof}

\begin{Lemma}
For $v \notin S$, $\frak{l}_v(C(k_v))$ and $\pr_1\frak{l}_v(C(k_v))$ are unramified in $H_v^\times/k_v^\times\partial F_v^\times$ and $F_v^\times/k_v^\times F_v^{\times\ell}$, respectively.
\end{Lemma}

\begin{proof}
Let $v \notin S$ and let $F'$ denote the splitting field of $X$. By Lemma \ref{unrKernel} it suffices to prove the statement for $\pr_1\frak{l}_v(C(k_v))$. Let $P \in  C(k_v)\setminus X$ be given by primitive integral coordinates. Then $\pr_1 \frak{l}_v(P) = \frak{l}_{1}(P) \in F_v^\times$. The algebra $F_v$ splits as a product of finite extensions of $k_v$. Since $v \nmid \ell$ it is enough to show that $\frak{l}_{1}(P)$ has valuation divisible by $\ell$ in each factor. If $K_w$ is some factor, we have an unramified tower of field extensions $k_v \subset K_w \subset F'_\frak{w}$, where $\frak{w} | w | v$ is some prime of $F'$. Over $F'$, $\frak{l}_1$ splits as a tuple of linear forms $\frak{l}_x$ with coefficients in $F'$, each defining the hyperosculating plane to $C$ at $x \in X$. Since the extensions are unramified it will suffice to show that
\[ \ord_{\frak{w}}(\frak{l}_x(P)) \equiv 0 \mod \ell\,, \text{ for each $x \in X$.} \] For this one makes use of the norm condition corresponding to the constant $c \in K^\times$. We have \[ \sum_{x \in X} \ord_\frak{w}(\frak{l}_x(P)) = \ord_\frak{w}\left( \prod_{x\in X} \frak{l}_x(P)\right) \equiv \ord_{\frak{w}} (N_{F/K}(\frak{l}_1(P))) \equiv \ord_\frak{w}(c) \equiv 0 \mod \ell\,.\]

It will suffice to show that at most one summand can be nonzero, for then all must be divisible by $\ell$. By assumption the special fiber of (an integral model for) $C \otimes_{k_v}F'_\frak{w}$ is nonsingular, so the points $x \mod \frak{w}$ are distinct. Each linear form $\frak{l}_x$ reduces mod $\frak{w}$ to define the hyperosculating plane to the special fiber at the point $x \mod \frak{w}$. So $\ord_{\frak{w}}(\frak{l}_x(P)) > 0$ if and only if $P \equiv x \mod \frak{w}$. Since the $x \in X$ have distinct reductions, this can occur for at most one $x$.
\end{proof}

\begin{Proposition}
\label{Selcontains}
\[\Sel_{alg}^{(\varphi)}(C/k) = \left\{ \Delta \in \mathcal{H}_k \,:\, \begin{array}{c} \text{ $\Delta$ is unramified outside $S$, }\\
 \forall v \in S\,,\, \res_v(\Delta) \in \frak{l}_v(C(k_v)) \end{array} \right\}\,.\]
 \end{Proposition}

\begin{proof}
Let $Z$ be the set in the proposition. It is clear from the preceding lemmas that $\Sel_{alg}^{(\varphi)}(C/k)$ is contained in $Z$. For the other containment, let $\Delta \in Z$ and $v \notin S$. We need to show that $\res_v(\Delta) \in \frak{l}_v(C(k_v))$. For any $\Delta' = \frak{l}_v(Q) \in \frak{l}_v(C(k_v))$, the ratio $\Delta/\Delta'$ is unramified and lies in $\mathcal{H}_{k_v}^0$. It follows from Lemma \ref{Pic0unr} that $\Delta/\Delta' = \frak{l}_v(P)$ for some $P \in E'(k_v)$. Since the descent map is affine we have $\Delta = \frak{l}_v(Q + P) \in \frak{l}_v(C(k_v))$.
\end{proof}

\begin{proof}[Proof of Theorem \ref{ComputeSel}]
Suppose $\Delta \in V_3$. Step (4) ensures that for all $v \in S$, $\res_v(\Delta) \in \frak{l}_vC(k_v)$. Moreover, since $S$ is not empty this also ensures that $\Delta$ represents a class in $\mathcal{H}_k$ (cf. Lemma \ref{rescheat}). By Lemma \ref{unrKernel}, step (1) ensures that $\Delta$ is unramified outside $S$. Proposition \ref{Selcontains} then shows that $\Delta$ represents a class in $\Sel_{alg}^{(\varphi)}(C/k)$. This shows that the algorithm computes a subset of the algebraic $\varphi$-Selmer set. The reverse containment follows from the fact (see (\ref{defFk})) that $\mathcal{H}_k$ is a subset of $\left\{ (\delta,\varepsilon) \in F^\times\times H_2^\times\,:\, \partial_2(\delta)=\beta\varepsilon^\ell \right\} /k^\times \partial F^{\times}$.
\end{proof}

\subsection{The Fake Selmer Set}
In practice it is often easier to compute the image of the algebraic Selmer set under the projection $\pr_1:\frac{H^\times}{k^\times\partial F^\times} \to \frac{F^\times}{k^\times F^{\times \ell}}$. Recall that the composition $\pr_1 \circ \frak{l} : \Pic(C) \to \frac{F^\times}{k^\times F^{\times \ell}}$ is given by $\frak{l}_1$.

\begin{Definition}
We define the {\em fake $\varphi$-Selmer set of $C$} to be
\[ \Sel_{fake}^{(\varphi)}(C/k) = \left\{ \delta \in \mathcal{F}_k  \,:\, \text{ for all primes $v$, } \res_v(\delta) \in \frak{l}_{1,v}(C(k_v)) \right\}\,.\]
\end{Definition}
The projection $\pr_1$ induces a map \[ \pr_1:\Sel_{alg}^{(\varphi)}(C/k) \To \Sel^{(\varphi)}_{fake}(C/k)\,.\] In general this can fail to be injective or surjective (it is possible to construct examples of both phenomena). Nevertheless, the fake Selmer set often yields useful information on the genuine Selmer set. For example, if the fake Selmer set is empty, then so is the Selmer set.

From (\ref{defFk}) it follows that the fake $\varphi$-Selmer set is contained in the sets
\begin{align*}
&\left\{ \delta \in \frac{F^\times}{k^\times F^{\times\ell}} \,:\, \begin{array}{c} \text{ $\delta$ is unramified outside $S$, }\\ \text{ $\partial_2(\delta) \in \beta H_2^{\times\ell}$ }\\
\forall v \in S\,,\, \res_v(\delta) \in {l}_{1,v}(C(k_v)) \end{array} \right\}\\
\subset &\left\{ \delta \in \frac{F^\times}{k^\times F^{\times\ell}} \,:\, \begin{array}{c} \text{ $\delta$ is unramified outside $S$, }\\ \text{ $N_{F/K}(\delta) \in ck^{\times\ell}$ }\\
\forall v \in S\,,\, \res_v(\delta) \in \frak{l}_{1,v}(C(k_v)) \end{array} \right\}\,.
\end{align*}
In particular, if either of these sets is empty, then so is the $\varphi$-Selmer set of $C$.

\section{Examples}

We have implemented the algorithm of Theorem \ref{ComputeSel} in the computer algebra system {\tt Magma} \cite{magma}, for $\ell = 3,5,7$ and $k = \Q$. This was used for all computations below.

\subsection{Proof of Theorem \ref{MainTheorem}}
As described in the introduction, the proof of Theorem \ref{MainTheorem} has been reduced to determination of $\ell$-primary parts of $\Sha$ for the $11$ isogeny classes in table \ref{cases}. For each the situation is similar. All are of rank $0$, the mod-$\ell$ Galois representation is reducible, and the optimal curve $E$ has a $\Q$-rational $\ell$-torsion point. The quotient by the cyclic subgroup generated by this point gives an $\ell$-isogeny $\varphi : E \rightarrow E'$. It is a well known result of Cassels \cite{CasselsVIII} that the $\ell$-part of the BSD conjecture is invariant under isogeny, so it suffices to determine the order of $\Sha(E/\Q)[\ell]$. The order predicted by the conjecture is $1$, while that of $\Sha(E'/\Q)$ is $\ell^2$.

Since $E(\Q) \simeq \Z/\ell\Z$ one can perform $\varphi$- and $\varphi'$-descents on $E'$ and $E$ more or less by hand using the method described in \cite{FisherThesis,FisherJEMS}. One gets that $\Sel^{(\varphi)}(E'/\Q) = 0$ while $\Sel^{(\varphi')}(E/\Q)$ is presented as a $3$-dimensional subgroup of $\Q^{\times}/\Q^{\times \ell}$. This establishes that $\Sha(E'/\Q)[\ell] \simeq \Z/\ell\Z \times \Z/\ell\Z$, however it is not sufficient to show that $\Sha(E/\Q)[\ell] = 0$ or, what is equivalent, that $\Sha(E'/\Q)[\ell^\infty] = \Sha(E'/\Q)[\ell]$.

The description in \cite{FisherThesis,FisherJEMS} also gives models for the elements of $\Sha(E'/\Q)[\varphi']$ as everywhere locally solvable genus one normal curves of degree $\ell$. A second $\varphi$-descent on any of the nontrivial elements $C$ gives $\Sel^{(\varphi)}(C/\Q) = \emptyset$ (in all $11$ cases it was in fact sufficient to compute fake Selmer sets). From this it follows that $\varphi(\Sha(E/\Q)[\ell]) = 0$ and hence that that $\Sha(E/\Q)[\ell] = 0$. Two explicit examples are given below; the computations for the others are similar. 

\begin{Remark}
Since the order of $\Sha(E'/\Q)[\varphi']/\varphi(\Sha(E/\Q)[\ell])$ must be a square, it suffices to do the second $\varphi$-descent on just one nontrivial element of $\Sha(E'/\Q)[\varphi']$.
\end{Remark}

\subsection{The pair $(1950y,5)$}
Let $E$ denote the elliptic curve labeled 1950y1 in Cremona's database. The Mordell-Weil group is cyclic of order $5$. Denote the corresponding isogeny by $\varphi:E \to E'$. The method for $\varphi$ and $\varphi'$-descents described in  \cite{FisherThesis,FisherJEMS} gives $\Sel^{(\varphi)}(E'/\Q)=0$ and an explicit isomorphism of $\Sel^{(\varphi')}(E/\Q)$ with the subgroup of $\Q^\times/\Q^{\times 5}$ generated by $2$, $3$ and $13$. The image of $E(\Q)$ under the connecting homomorphism is generated by $2\cdot3\cdot 13^2$. This implies that the $\F_5$-dimensions of $\Sha(E/\Q)[\varphi]$ and $\Sha(E'/\Q)[\varphi']$ are $0$ and $2$, respectively.

The classes of $2^{\pm 1}$ modulo $\Q^{\times 5}$ are represented by the genus one normal curve:
\[ C : \left\{ \begin{array}{cc}
2u_1u_2 - u_3u_5 - 6u_4^2 & =0\\
    4u_1u_5 - u_2u_4 - 6u_3^2 & =0\\
    13u_1u_4 - 6u_2u_3 + 2u_5^2 & =0\\
    13u_1u_3 + u_2^2 - 12u_4u_5 & =0\\
    26u_1^2 + u_2u_5 - 36u_3u_4 & =0\\ \end{array} \right\} \subset \BP^4\,.\]
This has good reduction outside the primes dividing $1950$. The action of $E'[\varphi'] = \mu_5$ on $C$ is, after identifying $Q \in E'[\varphi']$ with $\zeta \in \mu_5$, given by
\[ Q + (x_1:x_2:x_3:x_4:x_5) = (x_1:\zeta x_2:\zeta^2 x_3: \zeta^3 x_4: \zeta^4 x_5)\,. \] The quotient by this action gives $C$ the structure of a $\varphi'$-covering of $E$ (well defined up to composition with a translation by a point in $E(\Q) \simeq \Z/5\Z$). To prove that $\Sha(E/\Q)[5]=0$ it suffices to show that the $\varphi$-Selmer set of $C$ is empty.

The quotient evidently identifies the points lying on any given coordinate hyperplane. Let $F = \Q(\theta)$, where $\theta$ is a $5$-th root of $2$. The hyperplane $\{ x_1 = 0 \}$ intersects $C$ transversely at the point  $P = (0:-6\theta^3:-\theta^2:\theta:6)$ and at each of its $G_\Q$-conjugates. These $5$ points form a torsor $X$ under $E'[\varphi']=\mu_5$. They are flex points; the linear form \[ \frak{l}_1 = 1871u_1 + 330\theta u_2 + 1224\theta^2u_3 + 1224\theta^3u_4 + 330\theta^4u_5 \] defines a hyperplane meeting $C$ at $P$ with multiplicity $5$.

One can check that $N_{F/\Q}(\frak{l}_1) \equiv u_1^5$ modulo the homogeneous ideal of $C$, so the constant $c \in \Q^\times$ corresponding to our choice for $\frak{l}_1$ is $1$. $\frak{l}_1$ has good reduction at all primes, so the fake Selmer set is contained in the set \[ V := \left\{ \delta \in \frac{F^\times}{\Q^\times F^{\times 5}} \,:\, \begin{array}{c} \text{ $\delta$ is unramified outside $\{2,3,5,13\}$ }\\ \text{ and $N_{F/\Q}(\delta) \in \Q^{\times 5}$ }\end{array} \right\}\,. \]

$F$ has class number $1$, so to compute $V$ we only need generators of a subgroup of the $\{2,3,5,13\}$-unit group of $F$ of finite index prime to $5$. This can be achieved through standard algorithms. One finds that $V$ is a cyclic group of order $5$, generated by the unit $\alpha = \theta^3 + \theta^2 - 1$.

To cut this down any further we need to make use of the local conditions at the primes in $\{2,3,5,13\}$. First we consider $p = 3$. $E$ has split multiplicative reduction and the Tamagawa numbers of $E$ and $E'$ at $3$ satisfy $c_3(E)/c_3(E') = 5$. This implies that $E'(\Q_3)/\varphi(E(\Q_3)) = 0$ (see for example \cite[Section 3]{MillerStoll}). It follows that the local image $\frak{l}_{1,3}(C(\Q_3)) \subset F_3^\times/\Q_3^\times F_3^{\times 5}$ consists of a single element. The $\F_3$-point $(2:1:1:2:1)$ on $C$ is nonsingular. So it lifts to a $\Q_3$-point in the $3$-adic neighborhood \[ U = (2+O(3):1+O(3):1+O(3):2+O(3):1+O(3)) \subset \BP^4(\Q_3)\,.\] One can check that $\frak{l}_1(2,1,1,2,1)$ is a unit in $F_3$ (i.e. has valuation $0$ in each factor of $F_3$). Hence for every $P \in U$, the class of $\frak{l}_1(P)$ in $F_3/F_3^{\times 5}$ is the same. A direct computation shows that $\frak{l}_1(2,1,1,2,1) \equiv \alpha^2 \not\equiv 1 \mod \Q_3^\times F_3^{\times 5}$. It follows that the fake Selmer set must be contained in $\{ \alpha^2 \} \subset V$.

We now consider the local condition at $p = 5$. We have $\dim E'(\Q_5)/\varphi(E(\Q_5)) = 1$. As above we find neighborhoods
\begin{align*}
U_1 &:= ( 59 + O(5^3) : 65 + O(5^3) : 14 + O(5^3) : 49 + O(5^3) : 1 + O(5^3) ) \subset \BP^4(\Q_5)\,,\\
U_2 &:= ( 109+ O(5^3) :  10+ O(5^3) : 29+ O(5^3) :  89+ O(5^3) :  1 + O(5^3) ) \subset \BP^4(\Q_5) \,,
\end{align*}
which contain points in $C(\Q_5)$. Evaluating $\frak{l}_1$ at the coordinates of any point in either neighborhood gives a unit in $F_5$. Here $F_5/\Q_5$ is a totally ramified field extension, so the class of a unit modulo $5$-th powers is determined by it class modulo $5^3$. 
Hence, $\frak{l}_{1,5}$ is constant on these neighborhoods. If $a = \frak{l}_{1,5}(59, 65, 14, 49, 1)$ and $b = \frak{l}_{1,5}(109, 10, 29, 89, 1)$, then $\frak{l}_{1,5}(C(\Q_5))$ is the affine line in $F_5^\times/\Q_5^\times F_5^{\times 5}$ consisting of classes represented by some $a^m\cdot b^n$ with $m + n \mod 5 \equiv 1$. One can check that $\res_5(\alpha^2)$ does not lie on this line. This shows that the fake $\varphi$-Selmer set is empty. 

Specifically, this shows that there exists no $\varphi$-covering of $C$ that is locally solvable outside $\{2,13\}$ (since we didn't use the local conditions there). From Lemma \ref{phidiv} and remark \ref{phidivremark} it follows that $\Sha(E/\Q)[5]=0$ as expected.

\subsection{The pair $(1230k,7)$}
Let $E$ be the curve labeled $1230k1$ in Cremona's Database. As above $E(\Q)$ is cyclic of order $7$. We let $\varphi:E \to E'$ by the quotient of $E$ by $E(\Q)$. Fisher's method for $7$-isogeny descent gives an explicit isomorphism of the $\varphi'$-Selmer group of $E$ with the subgroup of $\Q^\times/\Q^{\times 7}$ generated by $\{2,3,5\}$. Up to sign, the class in  $\Sha(E'/\Q)[\varphi']$ corresponding $9\Q^{\times 7}$ is represented by the curve:

\[ C : \left\{ \begin{array}{cc}
5u_1^2 - 3u_3u_6 + u_2u_7 &=0 \\
10u_1^2 - 3u_4u_5 + 12u_2u_7 &=0 \\
6u_2^2 + u_1u_3 - u_4u_7 &=0 \\
2u_2^2 + 2u_1u_3 - u_5u_6 &=0 \\
6u_3^2 + u_2u_4 - 5u_1u_5 &=0 \\
3u_3^2 + 3u_2u_4 - 5u_6u_7 &=0 \\
u_4^2 + u_3u_5 - 10u_2u_6 &=0 \\
3u_4^2 + 18u_3u_5 - 50u_1u_7 &=0 \\
3u_5^2 + u_4u_6 - 10u_3u_7 &=0 \\
50u_1u_2 - 3u_5^2 - 6u_4u_6 &=0 \\
5u_1u_4 - 6u_6^2 - u_5u_7 &=0 \\
5u_2u_3 - u_6^2 - u_5u_7 &=0 \\
3u_2u_5 - u_1u_6 - 2u_7^2 &=0 \\
3u_3u_4 - 6u_1u_6 - 2u_7^2 &= 0 \end{array} \right\} \subset \BP^6\,.\]

The first $\varphi$-descent shows that this curve violates the Hasse principle. The second $\varphi$-descent shows that it does not lift to an element of order $7$ in $\Sha(E/\Q)$.

Any coordinate hyperplane intersects $C$ in $7$ distinct flexes. The most convenient to work with are those given by $P = (0:\theta^5:-\theta^4:-6\theta^3:6\theta^2:3\theta:-9)$ defined over $F = \Q(\theta)$, where $\theta$ is a $7$-th root of $9$. The $7$ possible choices correspond to the $7$ distinct lifts of the class represented by $C$ to the $\varphi'$-Selmer group.

The hyperplane defined by \[ \frak{l}_{1} = 250111u_1 - 209538\theta u_2 + 102354\theta^2u_3 - 29225\theta^3u_4 + 29225\theta^4u_5 - 34118\theta^5u_6 + 23282\theta^6u_7\,,\] meets $C$ at $P$ with multiplicity $7$. Modulo the homogeneous ideal of $C$ we have $N_{F/\Q}(\frak{l}_1) \equiv (41^{-2}u_1)^7$, so again we may take $c = 1$. We then compute the class group of $F$ (it is trivial) and generators for a finite, and prime to $7$, index subgroup of the $\{2,3,5,7,41\}$-unit group of $F$. Using these we determine representatives in $F^\times$ for the subset of the unramified outside $\{2,3,5,7,41\}$-subgroup of $F^\times/\Q^\times F^{\times 7}$ consisting of elements whose norm is a $7$-th power. This gives a $5$-dimensional space which contains the fake $\varphi$-Selmer set. As in the previous example, the local conditions for $p \in \{2,3,5,7,41\}$ can then be used to reduce this to the empty set, establishing that the $\varphi$-Selmer set of $C$ is empty as well.

\subsection{A full $5$-descent}
Consider the elliptic curve $E/\Q$ with Weierstrass equation
\[ y^2 + xy + y = x^3 + x^2 - 12241995603x + 781027222459441\,.\] An $L$-function computation shows that $E$ has rank $1$. Solving for $\Reg(E(\Q))\cdot\#\Sha(E/\Q)$ in the conjectural formula yields $242.0138...$. This suggests that any point of infinite order in $E(\Q)$ will be very large. We do first and second $5$-isogeny descents to compute a generating set for $E(\Q)$.

$E(\Q)$ contains the point $P = (-49091 : 35573052 : 1)$ of order $5$. The quotient of $E$ by the subgroup generated by $P$ gives a $5$-isogeny $\varphi: E \to E'$. The $\varphi$- and $\varphi'$-Selmer groups of $E'$ and $E$ have dimensions $0$ and $2$, respectively. This confirms the fact that $E$ and $E'$ have rank $1$ and that there is no nontrivial $5$-torsion in $\Sha$ for either curve. The genus one normal curve
\[ C : \left\{ \begin{array}{cc}
 163u_1u_5 - u_2u_4 - u_3^2&=0\\
 u_1u_3 + u_2^2 - 326u_4u_5&=0\\
 u_1^2 + 467u_2u_5 - 2u_3u_4&=0\\
 u_1u_2 - 467u_3u_5 - 2u_4^2&=0\\
 u_1u_4 - u_2u_3 + 76121u_5^2 &=0
\end{array} \right\} \subset \BP^4\,,\] together with the appropriate covering map $C \to E$, represents the image of a generator of the free part of $E(\Q)$ under the connecting homomorphism $E(\Q) \to \Sel^{(\varphi')}(E/\Q)$. We know $C(\Q)$ is nonempty, but a naive search still reveals no $\Q$-points. A $\varphi$-descent on $C$ computes that the algebraic $\varphi$-Selmer set has size $1$ (this is in agreement with the fact that $\varphi$-Selmer group of $E'$ is trivial). Using the method of section \ref{Geometry} we construct the corresponding $\varphi$-covering. To ensure the coefficients of our model are manageable we use the minimization and reduction algorithms for genus one normal curves implemented in {\tt Magma} by Fisher \cite{FisherMinRed5}. What we obtain is the curve $D \subset \BP^4$ with defining equations
{\tiny 
\begin{align*}
&5z_1z_2 - 2z_1z_3 + 3z_1z_4 + 4z_1z_5 - 5z_2^2 - 8z_2z_3 + 8z_2z_4 - 4z_2z_5 -
 5z_3z_4 + z_3z_5 + 6z_4^2 - 2z_4z_5 + z_5^2=0\,,\\
&4z_1z_2 + 4z_1z_3 + 3z_1z_5 - 4z_2^2 - 7z_2z_3 - z_2z_4 - 4z_3^2 - 8z_3z_5 +
     8z_4^2 - z_4z_5 + 4z_5^2=0\,,\\
&3z_1z_2 - 10z_1z_3 + z_1z_4 + 3z_1z_5 - 3z_2^2 + 6z_2z_3 - 6z_2z_4 + 3z_2z_5 -
     6z_3z_5 - z_4^2 - 3z_4z_5 + 4z_5^2=0\,,\\
&5z_1z_2 + 2z_1z_3 + 3z_1z_4 + 3z_1z_5 - z_2^2 + 4z_2z_3 + 3z_2z_4 - 8z_2z_5 +
     3z_3z_4 - 3z_3z_5 + 2z_4^2 - 6z_4z_5 - 2z_5^2=0\,,\\
&4z_1^2 + 9z_1z_2 - z_1z_3 - 5z_1z_4 + 2z_1z_5 - 9z_2^2 - z_2z_3 - z_2z_4 + 9z_2z_5
     - 3z_3z_5 - 2z_4^2 + 3z_4z_5 - 10z_5^2=0\,,
\end{align*}}
together with a degree $5$ covering map $D \stackrel{\pi}{\to} C$. In a couple minutes one finds the rational point \[Q = (8576638489: 4495315592: 7115424631: -2573365369:8465644680) \in D(\Q)\,.\] As expected the coordinates of the image $\pi(Q) \in C(\Q)$ have approximately $5$ times as many digits; they are
\begin{align*}
u_1 &= -47781179424001250276101034444306427793974640994508803\,,\\
u_2 &= -36805769809432466564750059701585425584354450037869765\,,\\
u_3 &= 11567437127698252390861515883750795832342708671332291\,,\\
u_4 &= 24705602119472788155755723752744787614294729359169672\,,\\
u_5 &= 99572421720530424479069471725920845332991347451591\,.
\end{align*}
The image of $Q$ under the composition $D \to C \to E$ has infinite order. One can check directly that the  canonical height of the point\footnote{For the reader's sake the actual coordinates are omitted. The naive logarithmic height of the $x$-coordinate is just over $100$. As one could have anticipated, this is approximately $10$ times that of the coordinates of $Q$ itself.} is $242.0138...$ and that it, together with the torsion point $P$, generates $E(\Q)$. Its image under $\varphi$ generates $E'(\Q)$ (since $E'(\Q)/\varphi(E(\Q)) \subset \Sel^{(\varphi)}(E'/\Q) = 0$). In particular, the smallest nontrivial point on $E'$ has canonical height $5\cdot242.0138... = 1210.069...\,$. So it would have been no easier to work on the isogenous curve.

Finding the generator via a $4$-descent on $E$ would likely have required searching for points on a $4$-covering up to the impractical naive height of $10^{13}$. One could potentially improve this by extending to either an $8$- or a $12$-descent (both implemented in {\tt Magma}). The later requires class and unit group computations in a number field of degree $8$ and relatively large discriminant. While not entirely infeasible this would take a significant amount of time, even without requiring that the computations be performed rigorously. The $8$-descent runs into problems factorizing a 1600 digit integer which plays much the same role as our constant $c \in \Q^\times$. By way of contrast, our computation was complete in about one minute (the majority of which was spent searching for points on $D$).


\end{document}